\documentclass{cstr}

\usepackage{amsmath,amsfonts,amssymb,latexsym,bm}
\usepackage{amsthm}
\usepackage{mathrsfs}
\usepackage[dvipdfmx]{graphicx}
\usepackage{subfig}
\usepackage{url}
\usepackage{cite}

\newtheorem{definition}{Definition}
\newtheorem{theorem}{Theorem}

\newtheorem{proposition}{Proposition}

\usepackage{algorithm, algorithmic}
\renewcommand{\algorithmicrequire}{{\bf Input:}}
\renewcommand{\algorithmicensure}{{\bf Output:}}

\title{
A map of contour integral-based eigensolvers
for solving generalized eigenvalue problems
}

\author[1,*]{Akira Imakura}
\author[2]{Lei Du}
\author[1,3]{Tetsuya Sakurai}

\affil[1]{University of Tsukuba, Japan}
\affil[2]{Dalian University of Technology, China.}
\affil[3]{JST/CREST}

\email{imakura@cs.tsukuba.ac.jp}

\begin{document}
\maketitle
\thispagestyle{titlepage}

\begin{abstract}
Recently, contour integral-based methods have been actively studied for solving interior eigenvalue problems that find all eigenvalues located in a certain region and their corresponding eigenvectors. 
In this paper, we reconsider the algorithms of the five typical contour integral-based eigensolvers from the viewpoint of projection methods, and then map the relationships among these methods.
From the analysis, we conclude that all contour integral-based eigensolvers can be regarded as projection methods and can be categorized based on their subspace used, the type of projection and the problem to which they are applied implicitly.
\end{abstract}

\section{Introduction}
In this paper, we consider computing all eigenvalues located in a certain region of a generalized eigenvalue problem and their corresponding eigenvectors:
\begin{equation}
	A {\bm x}_i = \lambda_i B {\bm x}_i, \quad
  {\bm x}_i \in \mathbb{C}^n \setminus \{ {\bm 0} \}, \quad
  \lambda_i \in \Omega \subset \mathbb{C},
	\label{eq:gep}
\end{equation}
where $A, B \in \mathbb{C}^{n \times n}$ and $zB-A$ are assumed as nonsingular for any $z$ on the boundary $\Gamma$ of the region $\Omega$.
Let $m$ be the number of target eigenvalues $\lambda_i \in \Omega$ (counting multiplicity) and $X_\Omega = [{\bm x}_i | \lambda_i \in \Omega]$ be a matrix whose columns are the target eigenvectors.
\par
In 2003, Sakurai and Sugiura proposed a powerful algorithm for solving the interior eigenvalue problem \eqref{eq:gep} \cite{Sakurai:2003}.
Their projection-type method uses certain complex moment matrices constructed by a contour integral.
The basic concept is to introduce the rational function 
\begin{equation}
         r(z) := \widetilde{\bm v}^{\rm H} (z B - A)^{-1} B {\bm v}, \quad
         {\bm v}, \widetilde{\bm v} \in \mathbb{C}^n \setminus \{ {\bm 0} \},
         \label{eq:moment}
\end{equation}
whose poles are the eigenvalues of the generalized eigenvalue problem: $A{\bm x}_i = \lambda_i B {\bm x}_i$, and then compute all poles located in $\Omega$ by Kravanja's algorithm \cite{Kravanja:1999}, which is based on Cauchy's integral formula.
\par
Kravanja's algorithm can be expressed as follows.
Let $\Gamma$ be a positively oriented Jordan curve, i.e., the boundary of $\Omega$.
We define complex moments $\mu_k$ as
\begin{equation*}
        \mu_k := \frac{1}{2 \pi {\rm i} } \oint_\Gamma z^k r(z) {\rm d}z, \quad
        k = 0, 1, \dots, 2M-1.
\end{equation*}
Then, all poles located in $\Omega$ of a meromorphic function $r(z)$ are the eigenvalues of the generalized eigenvalue problem
\begin{equation}
        H_M^< {\bm y}_i = \theta_i H_M {\bm y}_i,
        \label{eq:hankel_gep1}
\end{equation}
where $H_M, H_M^<$ are Hankel matrices:
\begin{equation*}
        H_M := \left(
        \begin{array}{cccc}
                \mu_0 & \mu_1 & \cdots & \mu_{M-1} \\
                \mu_1 & \mu_2 & \cdots & \mu_{M} \\
                \vdots & \vdots & \ddots & \vdots \\
                \mu_{M-1} & \mu_{M} & \cdots & \mu_{2M-2}
        \end{array}
        \right), \quad H_M^< := \left(
        \begin{array}{cccc}
                \mu_1 & \mu_2 & \cdots & \mu_{M} \\
                \mu_2 & \mu_3 & \cdots & \mu_{M+1} \\
                \vdots & \vdots & \ddots & \vdots \\
                \mu_{M} & \mu_{M+1} & \cdots & \mu_{2M-1}
        \end{array}
        \right).
\end{equation*}
Applying Kravanja's algorithm to the rational function \eqref{eq:moment}, the generalized eigenvalue problem \eqref{eq:gep} reduces to the generalized eigenvalue problem with the Hankel matrices \eqref{eq:hankel_gep1}.
This algorithm is called the SS--Hankel method.
\par
The SS--Hankel method has since been developed by several researchers.
The SS--RR method based on the Rayleigh--Ritz procedure increases the accuracy of the eigenpairs \cite{Sakurai:2007}.
Block variants of the SS--Hankel and SS--RR methods (known as the block SS--Hankel and the block SS--RR methods, respectively) improve stability of the algorithms \cite{Ikegami:2008, Ikegami:2010, Ikegami:2010b}.
The block SS--Arnoldi method based on the block Arnoldi method has also been proposed \cite{Imakura:2014}.
Different from these methods, Polizzi proposed the FEAST eigensolver for Hermitian generalized eigenvalue problems, which is based on an accelerated subspace iteration with the Rayleigh--Ritz procedure \cite{Polizzi:2009}.
Their original 2009 version has been further developed \cite{Polizzi:2014, Guttel:2015, Yin:2014}.
\par
Meanwhile, the contour integral-based methods have been extended to nonlinear eigenvalue problems.
Nonlinear eigensolvers are based on the block SS--Hankel \cite{Asakura:2010, Asakura:2009} and the block SS--RR \cite{Yokota:2013} methods and a different type of contour integral-based nonlinear eigensolver was proposed by Beyn \cite{Beyn:2012}, which we call the Beyn method.
More recently, an improvement of the Beyn method was proposed based on using the canonical polyadic decomposition \cite{Barel:2016}.
\par
For Hermitian case, i.e., $A$ is a Hermitian and $B$ is a Hermitian positive definite, there are several related works based on Chebyshev polynomial filtering \cite{Zhoua:2006,Fang:2012} and based on rational interpolation \cite{Austin:2015}.
Specifically, Austin et al. analyzed that the contour integral-based eigensolvers have strong relationship with rational interpolation established in \cite{Austin:2014}, and proposed a projection type method only with real poles \cite{Austin:2015}.
\par
In this paper, we reconsider the algorithms of typical contour integral-based eigensolvers of \eqref{eq:gep}, namely, the block SS--Hankel method \cite{Ikegami:2008, Ikegami:2010}, the block SS--RR method \cite{Ikegami:2010b}, the FEAST eigensolver \cite{Polizzi:2009}, the block SS--Arnoldi method \cite{Imakura:2014} and the Beyn method \cite{Beyn:2012} as projection methods.
We then analyze and map the relationships among these methods.
From the map of the relationships, we also provide error analyses of each method.
Here, we note that our analyses cover the case of Jordan blocks of the size larger than one and infinite eigenvalues (or even both).
%
\par
The remainder of this paper is organized as follows.
Sections~\ref{sec:methods} and \ref{sec:preparation} briefly describe the algorithms of the contour integral-based eigensolvers and analyze the properties of their typical matrices, respectively.
The relationships among these methods are analyzed and mapped in Section~\ref{sec:map}.
Error analyses of the methods are presented in Section~\ref{sec:error}, and numerical experiments are conducted in Section~\ref{sec:experiments}.
The paper concludes with Section~\ref{sec:conclusions}.
\par
Throughout, the following notations are used.
Let $V = [{\bm v}_1, {\bm v}_2, \ldots, {\bm v}_L] \in \mathbb{C}^{n \times L}$ and define the range space of the matrix $V$ by $\mathcal{R}(V) := {\rm span}\{ {\bm v}_1, {\bm v}_2, \ldots, {\bm v}_L \}$.
In addition, for $A \in \mathbb{C}^{n \times n}$, $\mathcal{K}_k^\square (A,V)$ and $\mathcal{B}_k^\square (A,V)$ are the block Krylov subspaces:
\begin{align*}
        & \mathcal{K}_k^\square (A,V) := \mathcal{R}([V, AV, A^2V, \dots, A^{k-1}V]), \\
        & \mathcal{B}_k^\square (A,V) := \left\{ \sum_{i=0}^{k-1} A^i V \alpha_i \middle| \alpha_i \in \mathbb{C}^{L \times L} \right\}.
\end{align*}
We also define a block diagonal matrix with block elements $D_i \in \mathbb{C}^{n_i \times n_i}$ constructed as follows:
\begin{equation*}
        \bigoplus_{i=1}^d D_i = D_1 \oplus D_2 \oplus \dots \oplus D_d = \left(
        \begin{array}{cccc}
                D_1 & & & \\
                & D_2 & & \\
                & & \ddots & \\
                & & & D_d
        \end{array}
        \right) \in \mathbb{C}^{n \times n},
\end{equation*}
where $n = \sum_{i=1}^d n_i$.
\section{Contour integral-based eigensolvers}
\label{sec:methods}
The contour integral-based eigensolvers reduce the target eigenvalue problem \eqref{eq:gep} to a different type of small eigenvalue problem.
In this section, we first describe the reduced eigenvalue problems and then introduce the algorithms of the contour integral-based eigensolvers.
\subsection{Theoretical preparation}
As a generalization of the Jordan canonical form to the matrix pencil, we have the following theorem.
\begin{theorem}[Weierstrass canonical form]
\label{thm:weierstrass}
Let $zB-A$ be regular.
Then, there exist nonsingular matrices $\widetilde{P}^{\rm H}, Q$ such that
\begin{equation*}
        \widetilde{P}^{\rm H} (zB-A) Q =
        \bigoplus_{i=1}^r \left( zI_{n_i}-J_{n_i}(\lambda_i) \right)
        \oplus \bigoplus_{i=r+1}^d \left( zJ_{n_i}(0)-I_{n_i} \right),
\end{equation*}
where $J_{n_i}(\lambda_i)$ is the Jordan block with $\lambda_i$,
\begin{equation*}
        J_{n_i}(\lambda) = \left(
        \begin{array}{cccc}
                \lambda_i & 1 & & \\
                 & \lambda_i & \ddots & \\
                 & & \ddots & 1 \\
                 & & & \lambda_i
        \end{array}
        \right) \in \mathbb{C}^{n_i \times n_i},
\end{equation*}
and $zJ_{n_i}(0)-I_{n_i}$ is the Jordan block with $\lambda=\infty$,
\begin{equation*}
        zJ_{n_i}(0)-I_{n_i} = \left(
        \begin{array}{cccc}
                -1 & z & & \\
                 & -1 & \ddots & \\
                 & & \ddots & z \\
                 & & & -1
        \end{array}
        \right) \in \mathbb{C}^{n_i \times n_i}.
\end{equation*}
\end{theorem}
\par
The generalized eigenvalue problem $A {\bm x}_i = \lambda_i B {\bm x}_i$ has $r$ finite eigenvalues $\lambda_i, i = 1, 2, \dots, r$ with multiplicity $n_i$ and $d-r$ infinite eigenvalues $\lambda_i, i = r+1, r+2, \dots, d$ with multiplicity $n_i$.
Let $\widetilde{P}_i$ and $Q_i$ be submatrices of $\widetilde{P}$ and $Q$, respectively, corresponding to the $i$-th Jordan block, i.e., $\widetilde{P} = [\widetilde{P}_1, \widetilde{P}_2, \dots, \widetilde{P}_d], Q = [Q_1, Q_2, \dots, Q_d]$.
Then, the columns of $\widetilde{P}_i$ and $Q_i$ are the left/right generalized eigenvectors, whose 1st columns are the corresponding left/right eigenvectors.
\par
Let $L, M \in \mathbb{N}$ be input parameters and $V \in \mathbb{C}^{n \times L}$ be an input matrix.
We also define $S \in \mathbb{C}^{n \times LM}$ and $S_k \in \mathbb{C}^{n \times L}$ as follows:
\begin{equation}
        S := [ S_0, S_1, \dots, S_{M-1}], \quad
        S_k := \frac{1}{2 \pi {\rm i} } \oint_\Gamma z^k (zB-A)^{-1} BV {\rm d}z.
        \label{eq:set_s}
\end{equation}
From Theorem~\ref{thm:weierstrass}, we have the following theorem \cite[Theorem~4]{Ikegami:2008, Ikegami:2010}.
\begin{theorem}
\label{thm:Sk=CkS0}
Let $\widetilde{Q}^{\rm H} = Q^{-1}$ and $\widetilde{Q}_i$ be a submatrix of $\widetilde{Q}$ corresponding to the $i$-th Jordan block, i.e., $\widetilde{Q} = [\widetilde{Q}_1, \widetilde{Q}_2, \dots, \widetilde{Q}_d]$.
Then, we have
\begin{equation*}
        S_k = Q_\Omega J_\Omega^k \widetilde{Q}_\Omega^{\rm H} V
        = (Q_\Omega J_\Omega \widetilde{Q}_\Omega^{\rm H})^k (Q_\Omega \widetilde{Q}_\Omega^{\rm H}V)
        = C_\Omega^k S_0, \quad
        C_\Omega = Q_\Omega J_\Omega \widetilde{Q}_\Omega^{\rm H},
\end{equation*}
where
\begin{equation*}
        Q_\Omega = [Q_i | \lambda_i \in \Omega], \quad
        \widetilde{Q}_\Omega = [\widetilde{Q}_i | \lambda_i \in \Omega], \quad
        J_\Omega = \bigoplus_{\lambda_i \in \Omega} J_{n_i}(\lambda_i).
\end{equation*}
\end{theorem}
Using Theorem~\ref{thm:Sk=CkS0}, we also have the following theorem.
\begin{theorem}
        \label{thm:span}
        Let $m$ be the number of target eigenvalues (counting multiplicity) and $X_\Omega := [{\bm x}_i | \lambda_i \in \Omega]$ be a matrix whose columns are the target eigenvectors.
        Then, we have
        \begin{equation*}
                \mathcal{R}(X_\Omega) \subset \mathcal{R}(Q_\Omega) = \mathcal{R}(S),
        \end{equation*}
        if and only if ${\rm rank}(S) = m$.
\end{theorem}
\begin{proof}
        From Theorem~\ref{thm:Sk=CkS0} and the definition of $S$, we have
        \begin{equation*}
                S = [S_0, S_1, \dots, S_{M-1}] = Q_\Omega Z
        \end{equation*}
        where
        \begin{equation*}
                Z := [(Q_\Omega^{\rm H} V), J_\Omega(Q_\Omega^{\rm H} V), \dots, J_\Omega^{M-1}(Q_\Omega^{\rm H} V)].
        \end{equation*}
        Since $Q_\Omega$ is full rank, ${\rm rank}(S) = {\rm rank}(Z)$ and $\mathcal{R}(Q_\Omega) = \mathcal{R}(S)$ is satisfied if and only if ${\rm rank}(S) = {\rm rank}(Z) = m$.
        From the definitions of $X_\Omega$ and $Q_\Omega$, we have $\mathcal{R}(X_\Omega) \subset \mathcal{R}(Q_\Omega)$.
        Therefore, Theorem~\ref{thm:span} is proven.
        \par
        Here, we note that ${\rm rank}(Z)=m$ is not always satisfied for $m \leq LM$ even if $Q_\Omega^{\rm H}V$ is full rank \cite{Gutknecht:2007}.
\end{proof}
\subsection{Introduction to contour integral-based eigensolvers}
The contour integral-based eigensolvers are mathematically designed based on Theorems~\ref{thm:Sk=CkS0} and \ref{thm:span}, then the algorithms are derived from approximating the contour integral \eqref{eq:set_s} using some numerical integration rule:
\begin{equation}
        \widehat{S} := [ \widehat{S}_0, \widehat{S}_1, \dots, \widehat{S}_{M-1}], \quad
        \widehat{S}_k := \sum_{j=1}^N \omega_j z_j^k (z_jB-A)^{-1} BV,
        \label{eq:numerical_integral}
\end{equation}
where $z_j$ is a quadrature point and $\omega_j$ is its corresponding weight.
\subsubsection{The block SS--Hankel method}
The block SS--Hankel method \cite{Ikegami:2008, Ikegami:2010} is a block variant of the SS--Hankel method.
Define the block complex moments $\mu_k^\square \in \mathbb{C}^{L \times L}$ by
\begin{equation*}
        \mu_k^\square := \frac{1}{2 \pi {\rm i} } \oint_\Gamma z^k \widetilde{V}^{\rm H} (zB-A)^{-1} BV {\rm d}z
        = \widetilde{V}^{\rm H} S_k,
\end{equation*}
where $\widetilde{V} \in \mathbb{C}^{n \times L}$, and the block Hankel matrices $H_M, H_M^< \in \mathbb{C}^{LM \times LM}$ are given by
\begin{equation*}
        H_M^\square := \left(
        \begin{array}{cccc}
                \mu_0^\square & \mu_1^\square & \cdots & \mu_{M-1}^\square \\
                \mu_1^\square & \mu_2^\square & \cdots & \mu_{M}^\square \\
                \vdots & \vdots & \ddots & \vdots \\
                \mu_{M-1}^\square & \mu_{M}^\square & \cdots & \mu_{2M-2}^\square
        \end{array}
        \right), \quad H_M^< := \left(
        \begin{array}{cccc}
                \mu_1^\square & \mu_2^\square & \cdots & \mu_{M}^\square \\
                \mu_2^\square & \mu_3^\square & \cdots & \mu_{M+1}^\square \\
                \vdots & \vdots & \ddots & \vdots \\
                \mu_{M}^\square & \mu_{M+1}^\square & \cdots & \mu_{2M-1}^\square
        \end{array}
        \right).
\end{equation*}
We then obtain the following theorem \cite[Theorem~7]{Ikegami:2008, Ikegami:2010}.
\begin{theorem}
        \label{thm:hankel}
        If ${\rm rank}(S)=m$, then the nonsingular part of the matrix pencil $zH_M^\square-H_M^{\square <}$ is equivalent to $zI-J_\Omega$.
\end{theorem}
\par
According to Theorem~\ref{thm:hankel}, the target eigenpairs $(\lambda_i,{\bm x}_i), \lambda_i \in \Omega$ can be obtained through the generalized eigenvalue problem
\begin{equation}
        H_M^{\square <} {\bm y}_i = \theta_i H_M^\square {\bm y}_i.
        \label{eq:hankel_gep}
\end{equation}
In practice, we approximate the block complex moments ${\mu}_k^\square \in \mathbb{C}^{L \times L}$ by the numerical integral \eqref{eq:numerical_integral} such that
\begin{equation*}
        \widehat{\mu}_k^\square := \sum_{j=1}^N \omega_j z_j^k \widetilde{V}^{\rm H} (z_jB-A)^{-1} BV
        = \widetilde{V}^{\rm H} \widehat{S}_k,
\end{equation*}
and set the block Hankel matrices $\widehat{H}_M^\square, \widehat{H}_M^{\square <} \in \mathbb{C}^{LM \times LM}$ as follows:
\begin{equation}
        \widehat{H}_M^\square := \left(
        \begin{array}{cccc}
                \widehat\mu_0^\square & \widehat\mu_1^\square & \cdots & \widehat\mu_{M-1}^\square \\
                \widehat\mu_1^\square & \widehat\mu_2^\square & \cdots & \widehat\mu_{M}^\square \\
                \vdots & \vdots & \ddots & \vdots \\
                \widehat\mu_{M-1}^\square & \widehat\mu_{M}^\square & \cdots & \widehat\mu_{2M-2}^\square
        \end{array}
        \right), \quad \widehat{H}_M^{\square <} := \left(
        \begin{array}{cccc}
                \widehat\mu_1^\square & \widehat\mu_2^\square & \cdots & \widehat\mu_{M}^\square \\
                \widehat\mu_2^\square & \widehat\mu_3^\square & \cdots & \widehat\mu_{M+1}^\square \\
                \vdots & \vdots & \ddots & \vdots \\
                \widehat\mu_{M}^\square & \widehat\mu_{M+1}^\square & \cdots & \widehat\mu_{2M-1}^\square
        \end{array}
        \right).
        \label{eq:hankel}
\end{equation}
To reduce the computational costs and improve the numerical stability, we also introduce a low-rank approximation with a numerical rank $\widehat{m}$ of $\widehat{H}_M^\square$ based on singular value decomposition:
\begin{equation*}
        \widehat{H}_M^\square = [U_{{\rm H}1}, U_{{\rm H}2}] \left[
                \begin{array}{ll}
                        \Sigma_{{\rm H}1} & O \\
                        O & \Sigma_{ {\rm H}2}
                \end{array}
        \right] \left[
                \begin{array}{ll}
                        W_{ {\rm H}1}^{\rm H} \\
                        W_{ {\rm H}2}^{\rm H}
                \end{array}
        \right] \approx U_{ {\rm H}1} \Sigma_{ {\rm H}1} W_{ {\rm H}1}^{\rm H}.
\end{equation*}
In this way, the target eigenvalue problem \eqref{eq:gep} reduces to an $\widehat{m}$ dimensional standard eigenvalue problem, i.e.,
\begin{equation*}
        U_{ {\rm H}1}^{\rm H} \widehat{H}_M^{\square <} W_{ {\rm H}1} \Sigma_{ {\rm H}1}^{-1} {\bm t}_i = \theta_i {\bm t}_i.
\end{equation*}
The approximate eigenpairs are obtained as $(\widetilde{\lambda}_i,\widetilde{\bm x}_i) = (\theta_i, \widehat{S} W_{ {\rm H}1}\Sigma_{\rm H1}^{-1} {\bm t}_i)$.
The algorithm of the block SS--Hankel method is shown in Algorithm~\ref{alg:ss-hankel}.
\begin{algorithm}[t]
\caption{The block SS--Hankel method}
\label{alg:ss-hankel}
\begin{algorithmic}[1]
  \REQUIRE $L, M, N \in \mathbb{N}, V, \widetilde{V} \in \mathbb{C}^{n \times L}, (z_j, \omega_j)$ for $j = 1, 2, \dots, N$
  \ENSURE Approximate eigenpairs $(\widetilde{\lambda}_i, \widetilde{\bm x}_i)$ for $i = 1, 2, \dots, \widehat{m}$
  \STATE Compute $\widehat{S}_k = \sum_{j=1}^{N} \omega_j z_j^k (z_jB-A)^{-1} BV$ and $\widehat{\mu}_k^\square = \widetilde{V}^{\rm H} \widehat{S}_k$
  \STATE Set $\widehat{S} = [\widehat{S}_0, \widehat{S}_1, \dots, \widehat{S}_{M-1}]$ and block Hankel matrices $\widehat{H}_M^\square, \widehat{H}_M^{\square <}$ by \eqref{eq:hankel}
  \STATE Compute SVD of $\widehat{H}_M^\square$: $\widehat{H}_M^\square= [U_{ {\rm H}1}, U_{ {\rm H}2}] [\Sigma_{ {\rm H}1}, O; O, \Sigma_{ {\rm H}2}] [W_{ {\rm H}1}, W_{ {\rm H}2}]^{\rm H}$
  \STATE Compute eigenpairs $(\theta_i, {\bm t}_i)$ of $U_{ {\rm H}1}^{\rm H}\widehat{H}_M^{\square <} W_{ {\rm H}1}^{\rm H} \Sigma_{ {\rm H}1}^{-1} {\bm t}_i = \theta_i {\bm t}_i$, \\ and compute $(\widetilde{\lambda}_i, \widetilde{\bm x}_i) = (\theta_i, \widehat{S} W_{ {\rm H}1}\Sigma_{\rm H1}^{-1} {\bm t}_i)$ for $i = 1, 2, \dots, \widehat{m}$
\end{algorithmic}
\end{algorithm}
\subsubsection{The block SS--RR method}
Theorem~\ref{thm:span} indicates that the target eigenpairs can be computed by the Rayleigh--Ritz procedure over the subspace $\mathcal{R}(S)$, i.e.,
\begin{equation*}
        S^{\rm H} A S {\bm t}_i = \theta_i S^{\rm H} B S {\bm t}_i.
\end{equation*}
\par
The above forms the basis of the block SS--RR method \cite{Ikegami:2010b}.
In practice, the Rayleigh--Ritz procedure uses the approximated subspace $\mathcal{R}(\widehat{S}) \approx \mathcal{R}(S)$ and a low-rank approximation of $\widehat{S}$:
\begin{equation*}
        \widehat{S} = [U_1, U_2] \left[
                \begin{array}{ll}
                        \Sigma_1 & O \\
                        O & \Sigma_2
                \end{array}
        \right] \left[
                \begin{array}{ll}
                        W_1^{\rm H} \\
                        W_2^{\rm H}
                \end{array}
        \right] \approx U_1 \Sigma_1 W_1^{\rm H}.
\end{equation*}
In this case, the reduced problem is given by
\begin{equation*}
        U_1^{\rm H}AU_1 {\bm t}_i = \theta_i U_1^{\rm H}BU_1 {\bm t}_i.
\end{equation*}
The approximate eigenpairs are obtained as $(\widetilde{\lambda}_i,\widetilde{\bm x}_i) = (\theta_i, U_1 {\bm t}_i)$.
The algorithm of the block SS--RR method is shown in Algorithm~\ref{alg:ss-rr}.
\begin{algorithm}[t]
\caption{The block SS--RR method}
\label{alg:ss-rr}
\begin{algorithmic}[1]
  \REQUIRE $L, M, N \in \mathbb{N}, V \in \mathbb{C}^{n \times L}, (z_j, \omega_j)$ for $j = 1, 2, \dots, N$
  \ENSURE Approximate eigenpairs $(\widetilde{\lambda}_i, \widetilde{\bm x}_i)$ for $i = 1, 2, \dots, \widehat{m}$
  \STATE Compute $\widehat{S}_k = \sum_{j=1}^{N} \omega_j z_j^k (z_jB-A)^{-1} BV$, and set $\widehat{S} = [\widehat{S}_0, \widehat{S}_1, \dots, \widehat{S}_{M-1}]$
  \STATE Compute SVD of $\widehat{S}$: $\widehat{S}= [U_1, U_2] [\Sigma_1, O; O, \Sigma_2] [W_1, W_2]^{\rm H}$
  \STATE Compute eigenpairs $({\theta}_i, {\bm t}_i)$ of $U_1^{\rm H}AU_1 {\bm t}_i = \theta_i U_1^{\rm H}BU_1 {\bm t}_i$, \\ and compute $(\widetilde{\lambda}_i, \widetilde{\bm x}_i) = (\theta_i, U_1 {\bm t}_i)$ for $i = 1, 2, \dots, \widehat{m}$
\end{algorithmic}
\end{algorithm}
\subsubsection{The FEAST eigensolver}
The algorithm of the accelerated subspace iteration with the Rayleigh--Ritz procedure for solving Hermitian generalized eigenvalue problems is given in Algorithm~\ref{alg:subiter}.
Here, $\rho(A,B)$ is called an accelerator.
When $\rho(A,B) = B^{-1}A$, Algorithm~\ref{alg:subiter} becomes the standard subspace iteration with the Rayleigh--Ritz procedure.
It computes the $L$ largest-magnitude eigenvalues and their corresponding eigenvectors.
\begin{algorithm}[t]
\caption{The accelerated subspace iteration with the Rayleigh--Ritz procedure}
\label{alg:subiter}
\begin{algorithmic}[1]
  \REQUIRE $L \in \mathbb{N}, V_0 \in \mathbb{C}^{n \times L}$
  \ENSURE Approximate eigenpairs $(\widetilde{\lambda}_i, \widetilde{\bm x}_i)$ for $i = 1, 2, \dots, L$
  \FOR{$k = 1, 2, \dots,$ until convergence} 
  \STATE Approximate subspace projection: $Q_{k} = \rho(A,B) \cdot V_{k-1}$
  \STATE Compute eigenpairs $(\theta_i^{(k)}, {\bm t}_i^{(k)})$ of $Q_k^{\rm H}AQ_k {\bm t}_i = \theta_i Q_k^{\rm H}BQ_k {\bm t}_i$, \\ and compute $(\widetilde{\lambda}_i^{(k)}, \widetilde{\bm x}_i^{(k)}) = (\theta_i^{(k)}, Q_k {\bm t}_i^{(k)})$ for $i = 1, 2, \dots, L$
  \STATE Set $V_{k} = [\widetilde{\bm x}_1^{(k)}, \widetilde{\bm x}_2^{(k)}, \dots, \widetilde{\bm x}_L^{(k)}]$
  \ENDFOR
\end{algorithmic}
\end{algorithm}
\par
The FEAST eigensolver \cite{Polizzi:2009}, proposed for Hermitian generalized eigenvalue problems, is based on an accelerated subspace iteration with the Rayleigh--Ritz procedure.
In the FEAST eigensolver, the accelerator $\rho(A,B)$ is set as
\begin{equation*}
        \rho(A,B) = \sum_{j=1}^N \omega_j (z_j B - A)^{-1} B
        \approx \frac{1}{2 \pi {\rm i} } \oint_\Gamma (zB-A)^{-1} B {\rm d}z,
\end{equation*}
based on Theorem~\ref{thm:span}.
Therefore, the FEAST eigensolver computes the eigenvalues located in $\Omega$ and their corresponding eigenvectors.
For numerical integration, the FEAST eigensolver uses the Gau\ss-Legendre quadrature or the Zolotarev quadrature; see \cite{Polizzi:2009, Guttel:2015}.
\par
In each iteration of the FEAST eigensolver, the target eigenvalue problem \eqref{eq:gep} is reduced to a small eigenvalue problem, i.e.,
\begin{equation*}
        \widehat{S}_0^{\rm H}A\widehat{S}_0 {\bm t}_i = \theta_i \widehat{S}_0^{\rm H}B\widehat{S}_0 {\bm t}_i,
\end{equation*}
based on the Rayleigh--Ritz procedure.
The approximate eigenpairs are obtained as $(\widetilde{\lambda}_i,\widetilde{\bm x}_i) = (\theta_i, \widehat{S}_0 {\bm t}_i)$.
The algorithm of the FEAST eigensolver is shown in Algorithm~\ref{alg:feast}.
\begin{algorithm}[t]
\caption{The FEAST eigensolver}
\label{alg:feast}
\begin{algorithmic}[1]
  \REQUIRE $L, N \in \mathbb{N}, V_0 \in \mathbb{C}^{n \times L}, (z_j, \omega_j)$ for $j = 1, 2, \dots, N$
  \ENSURE Approximate eigenpairs $(\widetilde{\lambda}_i, \widetilde{\bm x}_i)$ for $i = 1, 2, \dots, L$
  \FOR{$k = 1, 2, \dots,$ until convergence} 
  \STATE Compute $\widehat{S}_0^{(k)} = \sum_{j=1}^{N} \omega_j (z_jB-A)^{-1} BV_{k-1}$
  \STATE Compute eigenpairs $(\theta_i^{(k)}, {\bm t}_i^{(k)})$ of $\widehat{S}_0^{(k) {\rm H}}A\widehat{S}_0^{(k)} {\bm t}_i = \theta_i \widehat{S}_0^{(k) {\rm H}}B\widehat{S}_0^{(k)} {\bm t}_i$, \\ and compute $(\widetilde{\lambda}_i^{(k)}, \widetilde{\bm x}_i^{(k)}) = (\theta_i^{(k)}, \widehat{S}_0^{(k)} {\bm t}_i^{(k)})$ for $i = 1, 2, \dots, L$
  \STATE Set $V_k = [\widetilde{\bm x}_1^{(k)}, \widetilde{\bm x}_2^{(k)}, \dots, \widetilde{\bm x}_L^{(k)}]$
  \ENDFOR
\end{algorithmic}
\end{algorithm}
\subsubsection{The block SS--Arnoldi method}
From Theorems~\ref{thm:Sk=CkS0} and \ref{thm:span} and the definition of $C_\Omega:=Q_\Omega J_\Omega \widetilde{Q}_\Omega^{\rm H}$, we have the following three theorems \cite{Imakura:2014}.
\begin{theorem}
        \label{thm:arnoldi}
        The subspace $\mathcal{R}(S)$ can be expressed as the block Krylov subspace associated with the matrix $C_\Omega$:
        \begin{equation*}
                \mathcal{R}(S) = \mathcal{K}_M^\square(C_\Omega,S_0).
        \end{equation*}
\end{theorem}
\begin{theorem}
        \label{thm:arnoldi2}
        Let $m$ be the number of target eigenvalues (counting multiplicity).
        Then, if ${\rm rank}(S) = m$, the target eigenvalue problem \eqref{eq:gep} is equivalent to a standard eigenvalue problem of the form
\begin{equation}
        C_\Omega {\bm x}_i = \lambda_i {\bm x}_i, \quad
        {\bm x}_i \in \mathcal{R}(S) = \mathcal{K}_M^\square(C_\Omega,S_0).
        \label{eq:sep}
\end{equation}
\end{theorem}
\begin{theorem}
        \label{thm:mat-vec}
        Any $E_k \in \mathcal{B}_k^\square(C_\Omega,S_0)$ has the following formula: 
        \begin{equation*}
                E_k = \frac{1}{2 \pi {\rm i}} \oint_\Gamma \sum_{i=0}^{k-1} z^i (zB-A)^{-1} B V \alpha_i {\rm  d}z, \quad
                \alpha_i \in \mathbb{C}^{L \times L}.
        \end{equation*}
        Then, the matrix multiplication of $C_\Omega$ by $E_k$ becomes 
        \begin{equation*}
                C_\Omega E_k = \frac{1}{2 \pi {\rm i}} \oint_\Gamma z \sum_{i=0}^{k-1} z^i (zB-A)^{-1} B V \alpha_i {\rm  d}z.
        \end{equation*}
\end{theorem}
From Theorems~\ref{thm:arnoldi} and \ref{thm:arnoldi2}, we observe that the target eigenpairs $(\lambda_i,{\bm x}_i),\lambda_i \in \Omega$ can be computed by the block Arnoldi method with the block Krylov subspace $\mathcal{K}_M^\square(C_\Omega,S_0)$ for solving the standard eigenvalue problem \eqref{eq:sep}.
Here, we note that the matrix $C_\Omega$ is not explicitly constructed.
Instead, the matrix multiplication of $C_\Omega$ can be computed via the contour integral using Theorem~\ref{thm:mat-vec}.
By approximating the contour integral by a numerical integration rule, the algorithm of the block SS--Arnoldi method is derived (Algorithm~\ref{alg:ss-arnoldi}).
\par
A low-rank approximation technique to reduce the computational costs and improve stability is not applied in the current version of the block SS--Arnoldi method \cite{Imakura:2014}.
Improvements of the block SS--Arnoldi method has been developed in \cite{Imakura:2014c}.
\begin{algorithm}[t]
\caption{The block SS--Arnoldi method}
\label{alg:ss-arnoldi}
\begin{algorithmic}[1]
  \REQUIRE $L, M, N \in \mathbb{N}, V \in \mathbb{C}^{n \times L}, (z_j, \omega_j)$ for $j = 1, 2, \ldots, N$
  \ENSURE Approximate eigenpairs $(\widetilde{\lambda}_i, \widetilde{\bm x}_i)$ for $i = 1, 2, \ldots, LM$
  \STATE Solve $Y_j = (z_jB-A)^{-1} BV$ for $j = 1, 2, \ldots, N$
  \STATE $W_0 = \sum_{j=1}^{N} \omega_j Y_j$
  \STATE Compute QR decomposition of $W_0$: $W_0 = W_1 R$
  \STATE Set $\alpha_{1,j} = R^{-1}$ for $j = 1, 2, \ldots, N$
	\FOR{$k = 1, 2, \ldots, M$} 
    \STATE $\widetilde{\alpha}_{k,j} = z_j \alpha_{k,j}$ for $j = 1, 2, \ldots, N$
    \STATE $\widetilde{W}_k = \sum_{j=1}^{N} \omega_j Y_j \widetilde{\alpha}_{k,j}$
		\FOR{$i = 1, 2, \ldots, k$} 
      \STATE $H_{i,k} = W_i^{\rm H} \widetilde{W}_k$
      \STATE $\widetilde{\alpha}_{k,j} = \widetilde{\alpha}_{k,j} - \alpha_{i,j} H_{i,k}$ for $j = 1, 2, \ldots, N$
      \STATE $\widetilde{W}_k = \widetilde{W}_k - W_i H_{i,k}$
		\ENDFOR
    \STATE Compute QR decomposition of $\widetilde{W}_k$: $\widetilde{W}_k = W_{k+1} H_{k+1,k}$
    \STATE $\alpha_{k+1,j} = \widetilde{\alpha}_{k,j} H_{k+1,k}^{-1}$ for $j = 1, 2, \ldots, N$
	\ENDFOR
	\STATE Set $W = [W_1, W_2, \ldots, W_M]$ and $H_M = \{H_{i,j} \}_{1 \leq i, j \leq M}$
  \STATE Compute eigenpairs $(\theta_i, {\bm t}_i)$ of $H_M {\bm t}_i = \theta_i {\bm t}_i$,\\ and compute $(\widetilde{\lambda}_i,\widetilde{\bm x}_i) = (\theta_i, W {\bm t}_i)$ for $i = 1, 2, \ldots, LM$
\end{algorithmic}
\end{algorithm}
\subsubsection{The Beyn method}
The Beyn method is a nonlinear eigensolver based on the contour integral \cite{Beyn:2012}.
In this subsection, we consider the algorithm of the Beyn method for solving the generalized eigenvalue problem \eqref{eq:gep}.
\par
Let the singular value decomposition of $S_0$ be $S_0 = U_{0} \Sigma_{0} W_{0}^{\rm H}$, where $U_0 \in \mathbb{C}^{n \times m}, \Sigma_0 = {\rm diag}(\sigma_1, \sigma_2, \dots, \sigma_m), W_0 \in \mathbb{C}^{L \times m}$ and ${\rm rank}(S_0) = m$.
Then, from Theorem~\ref{thm:Sk=CkS0}, we have
\begin{equation}
        S_0 = Q_\Omega \widetilde{Q}_\Omega^{\rm H}V = U_0 \Sigma_0 W_0^{\rm H}, \quad
        S_1 = Q_\Omega J_\Omega \widetilde{Q}_\Omega^{\rm H}V.
        \label{eq:beyn1}
\end{equation}
Since $\mathcal{R}(Q_\Omega) = \mathcal{R}(U_0)$, we obtain
\begin{equation}
        Q_\Omega = U_0 Z, \quad Z = U_0^{\rm H} Q_\Omega \in \mathbb{C}^{m \times m},
        \label{eq:beyn2}
\end{equation}
where $Z$ is nonsingular.
With \eqref{eq:beyn1} and \eqref{eq:beyn2}, we find $U_0 Z \widetilde{Q}_\Omega^{\rm H} V = U_0 \Sigma_0 W_0^{\rm H}$ and thus $\widetilde{Q}_\Omega^{\rm H} V = Z^{-1} \Sigma_0 W_0^{\rm H}$.
This leads to
\begin{equation*}
        U_0^{\rm H} S_1 = Z J_\Omega \widetilde{Q}_\Omega^{\rm H} V = Z J_\Omega Z^{-1} \Sigma_0 W_0^{\rm H}.
\end{equation*}
Therefore, we have
\begin{equation*}
        Z J_\Omega Z^{-1} = U_0^{\rm H} S_1 W_0 \Sigma_0^{-1}.
\end{equation*}
This means that the target eigenpairs $(\lambda_i,{\bm x}_i), \lambda_i \in \Omega$ are computed by solving
%
%
\begin{equation*}
        U_0^{\rm H} S_1 W_0 \Sigma_0^{-1} {\bm t}_i = \theta_i {\bm t}_i,
\end{equation*}
where $(\lambda_i,{\bm x}_i) = (\theta_i, U_0 {\bm t}_i)$ \cite{Beyn:2012}.
\par
In practice, we compute a low-rank approximation of $\widehat{S}_0$ by the singular value decomposition, i.e., 
\begin{equation}
        \widehat{S}_0 = [U_{0,1}, U_{0,2}] \left[ 
                \begin{array}{ll}
                        \Sigma_{0,1} & O \\
                        O & \Sigma_{0,2}
                \end{array}
        \right] \left[
                \begin{array}{l}
                        W_{0,1}^{\rm H} \\
                        W_{0,2}^{\rm H}
                \end{array}
        \right] \approx U_{0,1} \Sigma_{0,1} W_{0,1}^{\rm H},
        \label{eq:beyn_svd}
\end{equation}
which reduces the target eigenvalue problem \eqref{eq:gep} to the standard eigenvalue problem
\begin{equation}
        U_{0,1}^{\rm H} \widehat{S}_1 W_{0,1} \Sigma_{0,1}^{-1} {\bm t}_i = \theta_i {\bm t}_i.
        \label{eq:beyn}
\end{equation}
The approximate eigenpairs are obtained as $(\widetilde{\lambda}_i,\widetilde{\bm x}_i) = (\theta_i, U_{0,1} {\bm t}_i)$.
The algorithm of the Beyn method for solving the generalized eigenvalue problem \eqref{eq:gep} is shown in Algorithm~\ref{alg:beyn}.
\begin{algorithm}[t]
\caption{The Beyn method}
\label{alg:beyn}
\begin{algorithmic}[1]
  \REQUIRE $L, N \in \mathbb{N}, V \in \mathbb{C}^{n \times L}, (z_j, \omega_j)$ for $j = 1, 2, \dots, N$
  \ENSURE Approximate eigenpairs $(\widetilde{\lambda}_i, \widetilde{\bm x}_i)$ for $i = 1, 2, \dots, \widehat{m}$
  \STATE Compute $\widehat{S}_0, \widehat{S}_1$, where $\widehat{S}_k = \sum_{j=1}^{N} \omega_j z_j^k (z_jB-A)^{-1} BV$
  \STATE Compute SVD of $\widehat{S}_0$: $\widehat{S}_0= [U_{0,1}, U_{0,2}] [\Sigma_{0,1}, O; O, \Sigma_{0,2}] [W_{0,1}, W_{0,2}]^{\rm H}$
  \STATE Compute eigenpairs $(\theta_i, {\bm t}_i)$ of $U_{0,1}^{\rm H}\widehat{S}_1W_{0,1} \Sigma^{-1}_{0,1} {\bm t}_i = \theta_i {\bm t}_i$,\\ and compute $(\widetilde{\lambda}_i, \widetilde{\bm x}_i) = (\theta_i, U_{0,1} {\bm t}_i)$ for $i = 1, 2, \dots, \widehat{m}$
\end{algorithmic}
\end{algorithm}
\section{Theoretical preliminaries for map building}
\label{sec:preparation}
As shown in Section~\ref{sec:methods}, contour integral-based eigensolvers are based on the property of the matrices $S$ and $S_k$ (Theorems~\ref{thm:Sk=CkS0} and \ref{thm:span}).
The practical algorithms are then derived by a numerical integral approximation.
As theoretical preliminaries for map building in Section~\ref{sec:map}, this section explores the properties of the approximated matrices $\widehat{S}$ and $\widehat{S}_k$. 
Here, we assume that $(z_j,\omega_j)$ satisfy the following condition:
\begin{equation}
        \sum_{j=1}^N \omega_j z_j^k \left\{ 
                \begin{array}{ll}
                        = 0, & k = 0, 1, \dots, N-2 \\
                        \neq 0, & k = -1
                \end{array}
                \right..
        \label{eq:weights}
\end{equation}
\par
If the matrix pencil $zB-A$ is diagonalizable and $(z_j,\omega_j)$ satisfies condition \eqref{eq:weights}, we have
\begin{equation*}
        \widehat{S}_k = C^k \widehat{S}_0, \quad
        C = X_r \Lambda_r \widetilde{X}_r^{\rm H},
\end{equation*}
where $X_r = [{\bm x}_1, {\bm x}_2, \ldots, {\bm x}_r]$ is a matrix whose columns are eigenvectors corresponding to finite eigenvalues, $\widetilde{X}_r = [\widetilde{\bm x}_1, \widetilde{\bm x}_2, \ldots, \widetilde{\bm x}_r]$ is a submatrix of $\widetilde{X} = X^{\rm -H}$: $\widetilde{X}_r^{\rm H} X_r = I$, and $\Lambda_r = {\rm diag}(\lambda_1, \lambda_2, \dots, \lambda_r)$; see \cite{Imakura:2015}.
In the following analysis, we introduce a similar relationship in the case that the matrix pencil $zB-A$ is non-diagonalizable.
First, we define an upper triangular Toeplitz matrix as follows.
\begin{definition}
\label{def:t}
For ${\bm a} = [a_1, a_2, \dots, a_n] \in \mathbb{C}^{1 \times n}$, define $T_n({\bm a})$ as an $n \times n$ triangular Toeplitz matrix, i.e.,
\begin{equation*}
        T_n({\bm a}) := \left(
        \begin{array}{cccc}
                a_1 & a_2 & \cdots & a_n \\
                0   & a_1 & \cdots & a_{n-1} \\
                \vdots & \ddots & \ddots & \vdots \\
                0 & \cdots & 0 & a_1
        \end{array}
        \right) \in \mathbb{C}^{n \times n}.
\end{equation*}
\end{definition}
\par
Let ${\bm a} = [a_1, a_2, \dots, a_n], {\bm b} = [b_1, b_2, \dots, b_n], {\bm c} = [c_1, c_2, \dots, c_n] \in \mathbb{C}^{1 \times n}$ and $\alpha, \beta \in \mathbb{C}$.
Then, we have
\begin{align}
        & \alpha T_n({\bm a}) + \beta T_n({\bm b}) = T_n(\alpha{\bm a} + \beta{\bm b}), \label{eq:t_linear} \\
        & T_n({\bm a}) T_n({\bm b}) = T_n({\bm c}), \quad c_i = \sum_{j=1}^i a_j b_{i-j+1}, \quad i = 1, 2, \dots, n \label{eq:t_mult}.
\end{align}
Letting ${\bm d} = [\alpha, \beta, 0, \dots, 0] \in \mathbb{C}^{1 \times n}$, we also have
\begin{align}
        & (T_n({\bm d}))^{k} = T_n({\bm d}^{(k)}), \quad {\bm d}^{(k)} = \left[ \binom k 0 \alpha^k \beta^0, \binom {k}{1} \alpha^{k-1} \beta^1, \dots, \binom {k}{n} \alpha^{k-n+1} \beta^{n-1} \right], \label{eq:t_power} \\
        & (T_n({\bm d}))^{-1} = T_n({\bm d}^{(-1)}), \quad {\bm d}^{(-1)} = \left[ \frac{1}{\alpha}, -\frac{\beta}{\alpha^2}, \dots, \frac{(-\beta)^{n-1}}{\alpha^n} \right]. \label{eq:t_inv}
\end{align}
\par
Using these relations \eqref{eq:t_linear}--\eqref{eq:t_inv}, we analyze properties of $\widehat{S}$ and $\widehat{S}_k$.
From Theorem~\ref{thm:weierstrass}, we have
\begin{align*}
        & (zB-A)^{-1} = Q 
        \left[ 
                \bigoplus_{i=1}^r \left( zI_{n_i}-J_{n_i}(\lambda_i) \right)^{-1}
                \oplus \bigoplus_{i=r+1}^d \left( zJ_{n_i}(0)-I_{n_i} \right)^{-1}
        \right]
        \widetilde{P}^{\rm H}, \\
        & B = P 
        \left[ 
                \bigoplus_{i=1}^r I_{n_i}
                \oplus \bigoplus_{i=r+1}^d J_{n_i}(0)
        \right]
        \widetilde{Q}^{\rm H},
\end{align*}
where $P := \widetilde{P}^{\rm -H}, \widetilde{Q}^{\rm H} := Q^{-1}$.
Therefore, the matrix $\widehat{S}_k$ can be written as
\begin{align}
        \widehat{S}_k &= \sum_{j=1}^N \omega_j z_j^k (z_j B - A)^{-1} B V \nonumber \\ 
        & = \sum_{j=1}^N \omega_j z_j^k Q \left\{ 
                \bigoplus_{i=1}^r \left( z_jI_{n_i}-J_{n_i}(\lambda_i) \right)^{-1} \right. \nonumber \\
                &\hphantom{= \sum_{j=1}^N \omega_j z_j^k Q \left\{ \right.} \quad \left. \oplus \bigoplus_{i=r+1}^d \left[ \left( z_jJ_{n_i}(0)-I_{n_i} \right)^{-1} J_{n_i}(0) \right]
        \right\} \widetilde{Q}^{\rm H} V \nonumber \\
        &= \left\{ \sum_{i=1}^r Q_i \left[ \sum_{j=1}^N \omega_j z_j^k \left( z_j I_{n_i} - J_{n_i}(\lambda_i) \right)^{-1} \right] \widetilde{Q}_i^{\rm H} V \right\} \nonumber \\
        &\phantom{=} \quad + \left\{ \sum_{i=r+1}^d Q_i \left[ \sum_{j=1}^N \omega_j z_j^k \left( z_j J_{n_i}(0) - I_{n_i} \right)^{-1} J_{n_i}(0) \right] \widetilde{Q}_i^{\rm H} V \right\}, 
        \label{eq:sk=sk1+sk2}
\end{align}
where $Q_i$ and $\widetilde{Q}_i$ are $n \times n_i$ submatrices of $Q$ and $\widetilde{Q}$ respectively, corresponding to the $i$-th Jordan block, i.e., $Q = [Q_1, Q_2, \dots, Q_d], \widetilde{Q} = [\widetilde{Q}_1, \widetilde{Q}_2, \dots, \widetilde{Q}_d]$.
\par
First, we consider the 1st term of $\widehat{S}_k$ \eqref{eq:sk=sk1+sk2}:
\begin{equation*}
        \widehat{S}_k^{(1)} := \sum_{i=1}^r Q_i \left[ \sum_{j=1}^N \omega_j z_j^k \left( z_j I_{n_i} - J_{n_i}(\lambda_i) \right)^{-1} \right] \widetilde{Q}_i^{\rm H} V.
\end{equation*}
From the relation
\begin{equation*}
        z_j I_{n_i}  - J_{n_i}(\lambda_i) = T_n([z_j-\lambda_i, -1, 0, \dots, 0])
\end{equation*}
and \eqref{eq:t_inv}, we have
\begin{equation*}
        (z_j I_{n_i} - J_{n_i}(\lambda_i))^{-1} = T_{n_i}\left(\left[\frac{1}{z_j-\lambda_i}, \frac{1}{(z_j-\lambda_i)^2}, \dots, \frac{1}{(z_j-\lambda_i)^{n_i}} \right]\right).
\end{equation*}
Thus, defining ${\bm f}_{k}(\lambda_i) \in \mathbb{C}^{1 \times n_i}$ as
\begin{equation}
        {\bm f}_{k}(\lambda_i) := \sum_{j=1}^N \omega_j z_j^k
        \left[\frac{1}{z_j-\lambda_i}, \frac{1}{(z_j-\lambda_i)^2}, \dots, \frac{1}{(z_j-\lambda_i)^{n_i}} \right],
        \label{eq:def_fk}
\end{equation}
from \eqref{eq:t_linear}, we have
\begin{equation}
        \sum_{j=1}^N \omega_j z_j^k (z_j I_{n_i}  - J_{n_i}(\lambda_i))^{-1} 
        = T_{n_i}\left({\bm f}_{k}(\lambda_i) \right).
        \label{eq:tfk}
\end{equation}
Therefore, $\widehat{S}^{(1)}_k$ can be rewritten as
\begin{equation}
        \widehat{S}^{(1)}_k = \sum_{i=1}^r Q_i T_{n_i}\left({\bm f}_{k}(\lambda_i) \right) \widetilde{Q}_i^{\rm H} V.
        \label{eq:sk12}
\end{equation}
\par
Here, the following propositions hold.
\begin{proposition}
\label{prop:w}
Suppose that $(\omega_j,z_j)$ satisfies condition \eqref{eq:weights}.
Then, for any $\lambda_i \neq 0$ and $0 \leq k \leq N+p-2, p \geq 1$, the relation
\begin{equation}
        \sum_{j=1}^N \frac{\omega_j z_j^k}{(z_j - \lambda_i)^p}
        = \lambda_i^k \sum_{j=1}^N \frac{\omega_j}{(z_j - \lambda_i)^p} \sum_{q=0}^{p-1} \binom k q \left( \frac{z_j - \lambda_i}{\lambda_i} \right)^q
        \label{eq:prop_weight}
\end{equation}
is satisfied.
\end{proposition}
\begin{proof}
Since $\lambda_i \neq 0$, we have
\begin{equation}
        \frac{\omega_j z_j^k}{(z_j - \lambda_i)^p} 
        = \frac{\omega_j}{(z_j-\lambda_i)^p} \lambda_i^k \left( \frac{z_j}{\lambda_i} \right)^k
        = \frac{\omega_j}{(z_j-\lambda_i)^p} \lambda_i^k \left( 1 + \frac{z_j - \lambda_i}{\lambda_i} \right)^k.
        \label{eq:w1}
\end{equation}
Here, from the binomial theorem $(a+b)^k = \sum_{q=0}^k \binom kq a^{k-q}b^q$, (\ref{eq:w1}) is rewritten as
\begin{equation*}
        \frac{\omega_j z_j^k}{(z_j - \lambda_i)^p} 
        = \frac{\omega_j}{(z_j-\lambda_i)^p} \lambda_i^k \sum_{q=0}^k \binom kq \left(\frac{z_j - \lambda_i}{\lambda_i} \right)^q.
\end{equation*}
Therefore, the left-hand side of (\ref{eq:prop_weight}) is 
\begin{align}
        &\sum_{j=1}^N \frac{\omega_j z_j^k}{(z_j - \lambda_i)^p}  \nonumber \\
        &\quad = \sum_{j=1}^N \frac{\omega_j}{(z_j-\lambda_i)^p} \lambda_i^k \sum_{q=0}^k \binom kq \left(\frac{z_j - \lambda_i}{\lambda_i} \right)^q  \nonumber \\
        &\quad = \sum_{j=1}^N \frac{\omega_j}{(z_j-\lambda_i)^p} \lambda_i^k \left[ \sum_{q=0}^{p-1} \binom kq \left(\frac{z_j - \lambda_i}{\lambda_i} \right)^q + \sum_{q=p}^{k} \binom kq \left(\frac{z_j - \lambda_i}{\lambda_i} \right)^q \right] \nonumber \\
        &\quad = \left[ \lambda_i^k \sum_{q=0}^{p-1} \binom kq \lambda_i^{-q} \sum_{j=1}^N \omega_j (z_j - \lambda_i)^{q-p} \right] \nonumber \\
        &\phantom{\quad = } \quad + \left[ \lambda_i^k \sum_{q=p}^k \binom kq \lambda_i^{-q} \sum_{j=1}^N \omega_j (z_j - \lambda_i)^{q-p} \right].
        \label{eq:left}
\end{align}
Because condition \eqref{eq:weights} is satisfied, we have
\begin{equation*}
        \sum_{j=1}^N \omega_j (z_j - \lambda_i)^{q-p} = 0, \quad
        q = p, p+1, \dots, N+p-2,
\end{equation*}
thus, for $k \leq N+p-2$, the 2nd term of \eqref{eq:left} becomes 0.
Therefore, for $k = 0, 1, \ldots, N+p-2$, we obtain
\begin{equation*}
        \sum_{j=1}^N \frac{\omega_j z_j^k}{(z_j - \lambda_i)^p} 
         = \sum_{j=1}^N \frac{\omega_j}{(z_j-\lambda_i)^p} \lambda_i^k \sum_{q=0}^{p-1} \binom kq \left(\frac{z_j - \lambda_i}{\lambda_i} \right)^q,
\end{equation*}
which proves Proposition~\ref{prop:w}.
\end{proof}
\begin{proposition}
\label{prop:jk}
Suppose that $(\omega_j,z_j)$ satisfies condition \eqref{eq:weights}.
Then, for any $0 \leq k \leq N-1$, the relation
\begin{equation}
        T_{n_i}\left( {\bm f}_k(\lambda_i) \right)
        = (J_{n_i}(\lambda_i))^k T_{n_i}\left({\bm f}_0(\lambda_i) \right)
        \label{eq:tk=jkt0}
\end{equation}
is satisfied, where ${\bm f}_k(\lambda_i)$ is defined by \eqref{eq:def_fk} and $0^0 = 1$.
\end{proposition}
\begin{proof}
We first consider the case of $\lambda_i=0$.
From $J_{n_i}(0) = T_{n_i}([0, 1, 0, \ldots, 0])$, there exists a vector ${\bm t}_{0,k} \in \mathbb{C}^{1 \times n_i}$ satisfying
\begin{equation*}
        T_{n_i}({\bm t}_{0,k}) := (J_{n_i}(0))^k T_{n_i}\left({\bm f}_0(0) \right).
\end{equation*}
Then, from \eqref{eq:t_mult} and \eqref{eq:t_power}, the $p$-th element $( {\bm t}_{0,k} )_p$ of ${\bm t}_{0,k}$ can be written as
\begin{equation*}
        ( {\bm t}_{0,k} )_p = \left\{
                \begin{array}{ll}
                        0 & (1 \leq p \leq k) \\
                        \sum_{j=1}^N \omega_j z_j^{k-p} & (k < p \leq n_i)
                \end{array}
        \right..
\end{equation*}
On the other hand, from the definition \eqref{eq:def_fk}, the vector ${\bm f}_k(0)$ can be written as
\begin{equation*}
        {\bm f}_k(0) = \sum_{j=1}^N \omega_j [z_j^{k-1}, z_j^{k-2}, \ldots, z_j^{k-n_i}].
\end{equation*}
Since the first $k$ elements of ${\bm f}_k(0)$ are 0 from condition \eqref{eq:weights}, we have ${\bm f}_k(0) = {\bm t}_{0,k}$.
Therefore, \eqref{eq:tk=jkt0} is satisfied for $\lambda_i = 0$.
\par
Next, we consider the case of $\lambda_i \neq 0$.
From $J_{n_i}(\lambda_i) = T_{n_i}([\lambda_i,1,0,\dots,0])$ and \eqref{eq:t_power}, we have
\begin{equation}
        (J_{n_i}(\lambda_i))^k = T_{n_i}\left(\left[\lambda_i^k, \binom k1 \lambda_i^{k-1}, \dots, \binom k {n_i} \lambda_i^{k-n_i+1} \right]\right).
        \label{eq:jk}
\end{equation}
Let ${\bm t}_k \in \mathbb{C}^{1 \times n_i}$ be a vector satisfying
\begin{equation*}
        T_{n_i}({\bm t}_k) := (J_{n_i}(\lambda_i))^k T_{n_i}\left( {\bm f}_0(\lambda_i) \right).
\end{equation*}
Then, from \eqref{eq:t_mult}, \eqref{eq:jk} and the definition of ${\bm f}_k(\lambda_i)$ \eqref{eq:def_fk}, the $p$-th element $({\bm t}_k)_p$ of ${\bm t}_k$ can be written as
\begin{align*}
        ({\bm t}_k)_p
        &= \sum_{q = 1}^p \binom k {q-1} \lambda_i^{k-q+1} \sum_{j=1}^N \omega_j \frac{1}{(z_j-\lambda_i)^{p-q+1}} \nonumber \\
        &= \lambda_i^k \sum_{j=1}^N \frac{\omega_j}{(z_j-\lambda_i)^p} \sum_{q = 1}^p \binom k {q-1} \frac{(z_j-\lambda_i)^{q-1}}{ \lambda_i^{q-1}} \nonumber \\
        &= \lambda_i^k \sum_{j=1}^N \frac{\omega_j}{(z_j-\lambda_i)^p} \sum_{q = 0}^{p-1} \binom k {q} \left( \frac{z_j-\lambda_i}{ \lambda_i} \right)^q.
        \label{eq:tp}
\end{align*}
By Proposition~\ref{prop:w}, for $0 \leq k \leq N+p-2$, we obtain
\begin{equation*}
        ({\bm t}_k)_p = \sum_{j=1}^N \frac{\omega_j z_j^k}{(z_j - \lambda_i)^p}.
\end{equation*}
Therefore, for $0 \leq k \leq N-1$, we have
\begin{equation*}
        {\bm t}_k 
        = \left[ \sum_{j=1}^N \frac{\omega_j z_j^k}{z_j-\lambda_i}, \sum_{j=1}^N \frac{\omega_j z_j^k}{(z_j-\lambda_i)^2}, \dots, \sum_{j=1}^N \frac{\omega_j z_j^k}{(z_j-\lambda_i)^{n_i}} \right]
        = {\bm f}_k(\lambda_i),
\end{equation*}
and 
\begin{equation*}
        T_{n_i}\left( {\bm f}_k(\lambda_i) \right)
        = T_{n_i}\left( {\bm t}_k \right)
        = (J_{n_i}(\lambda_i))^k T_{n_i}\left({\bm f}_0(\lambda_i) \right)
\end{equation*}
is satisfied, proving Proposition~\ref{prop:jk}.
\end{proof}
%
%
%
From Proposition~\ref{prop:jk}, by substituting \eqref{eq:tk=jkt0} into \eqref{eq:sk12} and using \eqref{eq:tfk}, we obtain
\begin{align}
        \widehat{S}_k^{(1)} 
        &= \sum_{i=1}^r Q_i \left( J_{n_i}(\lambda_i) \right) ^k T_{n_i}({\bm f}_0(\lambda_i)) \widetilde{Q}_i^{\rm H} V \nonumber \\
        &= \sum_{i=1}^r Q_i \left( J_{n_i}(\lambda_i) \right) ^k \left[ \sum_{j=1}^N \omega_j \left( z_j I_{n_i} - J_{n_i}(\lambda_i) \right)^{-1} \right] \widetilde{Q}_i^{\rm H} V
        \label{eq:sk1}
\end{align}
%
for any $0 \leq k \leq N - 1$.
\par
Now consider the 2nd term of $\widehat{S}_k$ \eqref{eq:sk=sk1+sk2}, i.e.,
\begin{equation*}
        \widehat{S}_k^{(2)} := \sum_{i=r+1}^d Q_i \left[ \sum_{j=1}^N \omega_j z_j^k \left( z_j J_{n_i}(0) - I_{n_i} \right)^{-1} J_{n_i}(0) \right] \widetilde{Q}_i^{\rm H} V.
\end{equation*}
From the relations
\begin{align*}
        &z_j J_{n_i}(0) - I_{n_i} = T_{n_i}([-1, z_j, 0, \dots, 0]), \\
        &J_{n_i}(0) = T_{n_i}([0, 1, 0, \dots, 0])
\end{align*}
and \eqref{eq:t_mult} and \eqref{eq:t_inv}, we have
\begin{equation*}
        \left(z_j J_{n_i}(0) - I\right)^{-1} J_{n_i}(0) 
        = - T_{n_i}([0, 1, z_j, z_j^2, \dots, z_j^{n_i-2}]).
\end{equation*}
In addition, from \eqref{eq:t_linear}, we have
\begin{align*}
        & \sum_{j=1}^N \omega_j z_j^k \left(z_j J_{n_i}(0) - I_{n_i} \right)^{-1} J_{n_i}(0) \\
        & \quad = - T_{n_i} \left( \left[ 0, \sum_{j=1}^N \omega_j z_j^k, \sum_{j=1}^N \omega_j z_j^{k+1}, \dots, \sum_{j=1}^N \omega_j z_j^{k+n_i-2}\right]\right).
\end{align*}
Here, because $(z_j,\omega_j)$ satisfies condition \eqref{eq:weights},
\begin{equation*}
        \sum_{j=1}^N \omega_j z_j^k \left(z_j J_{n_i}(0) - I_{n_i} \right)^{-1} J_{n_i}(0) = O
\end{equation*}
is satisfied for any $0 \leq k \leq N-n_i$.
Therefore, letting
\begin{equation*}
        \eta := \max_{r+1 \leq i \leq d} n_i,
\end{equation*}
the 2nd term of $\widehat{S}_k$ is $O$ for any $0 \leq k \leq N - \eta$, i.e.,
\begin{equation}
        \widehat{S}_k^{(2)} = O.
        \label{eq:sk2}
\end{equation}
\par
From \eqref{eq:sk1} and \eqref{eq:sk2}, we have the following theorems.
\begin{theorem}
\label{thm:hat_Sk=CkS0}
%
%
%
Suppose that $(\omega_j,z_j)$ satisfies condition \eqref{eq:weights}.
Then, we have
\begin{equation*}
        \widehat{S}_k = C^k \widehat{S}_0, \quad
        C = Q_{1:r} J_{1:r} \widetilde{Q}_{1:r}^{\rm H},
\end{equation*}
for any $0 \leq k \leq N-\eta$, where
\begin{equation*}
        Q_{1:r} := [Q_1, Q_2, \dots, Q_r], \quad
        \widetilde{Q}_{1:r} := [\widetilde{Q}_1, \widetilde{Q}_2, \dots, \widetilde{Q}_r], \quad
        J_{1:r} := \bigoplus_{i=1}^r J_{n_i}(\lambda_i).
\end{equation*}
\end{theorem}
\begin{proof}
From \eqref{eq:sk1} and \eqref{eq:sk2}, we have
\begin{equation*}
        \widehat{S}_k = \sum_{i=1}^r Q_i (J_{n_i}(\lambda_i))^k \left[ \sum_{j=1}^N \omega_j \left( z_j I_{n_i} - J_{n_i}(\lambda_i) \right)^{-1} \right] \widetilde{Q}_i^{\rm H} V,
\end{equation*}
for any $0 \leq k \leq N-\eta$.
Here, we let
\begin{align*}
        &F_{n_i} := T_{n_i}\left({\bm f}_0(\lambda_i)\right) = \sum_{j=1}^N \omega_j \left( z_j I_{n_i} - J_{n_i}(\lambda_i) \right)^{-1}, \\
        &F_{1:r} := \bigoplus_{i=1}^r F_{n_i},
\end{align*}
then we obtain
\begin{align*}
        \widehat{S}_k 
        &= \sum_{i=1}^r Q_i \left( J_{n_i}(\lambda_i) \right)^k F_{n_i} \widetilde{Q}_i^{\rm H} V \nonumber \\
        &= Q_{1:r} J_{1:r}^k F_{1:r} \widetilde{Q}_{1:r}^{\rm H} V \nonumber \\
        &= (Q_{1:r} J_{1:r} \widetilde{Q}_{1:r}^{\rm H})^k (Q_{1:r} F_{1:r} \widetilde{Q}_{1:r}^{\rm H} V) \nonumber \\
        &= C^k \widehat{S}_0.
\end{align*}
Therefore, Theorem~\ref{thm:hat_Sk=CkS0} is proven.
\end{proof}
\begin{theorem}
\label{thm:eigenproblem}
If $(z_j,\omega_j)$ satisfies condition \eqref{eq:weights}, then the standard eigenvalue problem
\begin{equation}
        C {\bm x}_i = \lambda_i {\bm x}_i, \quad
        {\bm x}_i \in \mathcal{R}(Q_{1:r}), \quad
        \lambda_i \in \Omega \subset \mathbb{C}, 
        \label{eq:sep_r}
\end{equation}
is equivalent to the generalized eigenvalue problem \eqref{eq:gep}.
\end{theorem}
\begin{proof}
From the definition of $C := Q_{1:r} J_{1:r} \widetilde{Q}_{1:r}^{\rm H}$, the matrix $C$ has the same right eigenpairs $(\lambda_i,{\bm x}_i), i = 1, 2, \dots, r$ as the matrix pencil $zB-A$, i.e., ${\bm x}_i \in \mathcal{R}(Q_{1:r})$.
The other eigenvalues of $C$ are $0$, and their corresponding eigenvectors are equivalent to the right eigenvectors associated with the infinite eigenvalues $\lambda_i = \infty$ of $zB-A$, i.e., ${\bm x}_i \not\in \mathcal{R}(Q_{1:r})$.
Therefore, Theorem~\ref{thm:eigenproblem} is proven.
\end{proof}
\section{Map of the relationships among contour integral-based eigensolvers}
\label{sec:map}
Section~\ref{sec:preparation} analyzed the properties of the approximated matrices $\widehat{S}$ and $\widehat{S}_k$ (Theorem~\ref{thm:hat_Sk=CkS0}) and introduced the standard eigenvalue problem \eqref{eq:sep_r} equivalent to the target eigenvalue problem \eqref{eq:gep} (Theorem~\ref{thm:eigenproblem}).
\par
In this section, based on Theorems~\ref{thm:hat_Sk=CkS0} and \ref{thm:eigenproblem}, we reconsider the algorithms of the contour integral-based eigensolvers in terms of projection methods and map the relationships, focusing on their subspace used, the type of projection and the problem to which they are applied implicitly.
\subsection{Reconsideration of the contour integral-based eigensolvers}
As described in Section~\ref{sec:methods}, the subspaces $\mathcal{R}(S)$ and $\mathcal{R}(S_k)$ contain only the target eigenvectors ${\bm x}_i, \lambda_i \in \Omega$ based on Cauchy's integral formula.
In contrast, the subspaces $\mathcal{R}(\widehat{S})$ and $\mathcal{R}(\widehat{S}_k)$ are rich in the component of the target eigenvectors as will be shown in Section~\ref{sec:error}.
\subsubsection{The block SS--RR method and the FEAST eigensolvers}
The block SS--RR method and the FEAST eigensolvers are easily reconfigured as projection methods.
\par
The block SS--RR method solves $A{\bm x}_i = \lambda_i B {\bm x}_i$ through the Rayleigh--Ritz procedure on $\mathcal{R}(\widehat{S})$.
The block SS--RR method (Algorithm~\ref{alg:ss-rr}) is derived using a low-rank approximation of the matrix $\widehat{S}$ as shown in Section~2.2.
Since $\mathcal{R}(\widehat{S})$ is rich in the component of the target eigenvectors, the target eigenpairs are well approximated by the Rayleigh--Ritz procedure.
\par
The FEAST eigensolver conducts accelerated subspace iteration with the Rayleigh--Ritz procedure.
In each iteration of the FEAST eigensolver, the Rayleigh--Ritz procedure solves $A{\bm x}_i = \lambda_i B {\bm x}_i$ on $\mathcal{R}(\widehat{S}_0)$.
Like $\mathcal{R}(\widehat{S})$ in the block SS--RR method, $\mathcal{R}(\widehat{S}_0)$ is rich in the component of the target eigenvectors; therefore, the FEAST eigensolver also well approximates the target eigenpairs by the Rayleigh--Ritz procedure.
\subsubsection{The block SS--Hankel method, the block SS--Arnoldi method and the Beyn method}
From Theorem~\ref{thm:hat_Sk=CkS0}, we rewrite the block complex moments $\widehat{\mu}_k^\square$ of the block SS--Hankel method as 
\begin{equation*}
        \widehat{\mu}_k^\square = \widetilde{V}^{\rm H} \widehat{S}_k = \widetilde{V}^{\rm H} C \widehat{S}_{k-1} = \cdots = \widetilde{V}^{\rm H} C^k \widehat{S}_0.
\end{equation*}
Thus, the block Hankel matrices $\widehat{H}_M^\square, \widehat{H}_M^{\square <}$ \eqref{eq:hankel} become
\begin{align*}
        & \widehat{H}_M^\square = \left(
        \begin{array}{cccc}
                \widetilde{V}^{\rm H} \widehat{S}_0 & \widetilde{V}^{\rm H} \widehat{S}_1 & \cdots & \widetilde{V}^{\rm H} \widehat{S}_{M-1} \\
                \widetilde{V}^{\rm H} C \widehat{S}_0 & \widetilde{V}^{\rm H} C \widehat{S}_1 & \cdots & \widetilde{V}^{\rm H} C \widehat{S}_{M-1} \\
                \vdots & \vdots & \ddots & \vdots \\
                \widetilde{V}^{\rm H} C^{M-1} \widehat{S}_0 & \widetilde{V}^{\rm H} C^{M-1} \widehat{S}_1 & \cdots & \widetilde{V}^{\rm H} C^{M-1} \widehat{S}_{M-1}
        \end{array}
        \right), \\
        & \widehat{H}_M^{\square <} = \left(
        \begin{array}{cccc}
                \widetilde{V}^{\rm H} C \widehat{S}_0 & \widetilde{V}^{\rm H} C \widehat{S}_1 & \cdots & \widetilde{V}^{\rm H} C \widehat{S}_{M-1} \\
                \widetilde{V}^{\rm H} C^2 \widehat{S}_0 & \widetilde{V}^{\rm H} C^2 \widehat{S}_1 & \cdots & \widetilde{V}^{\rm H} C^2 \widehat{S}_{M-1} \\
                \vdots & \vdots & \ddots & \vdots \\
                \widetilde{V}^{\rm H} C^{M} \widehat{S}_0 & \widetilde{V}^{\rm H} C^{M} \widehat{S}_1 & \cdots & \widetilde{V}^{\rm H} C^{M} \widehat{S}_{M-1}
        \end{array}
        \right),
\end{align*}
respectively.
Here, let
\begin{equation*}
        \widetilde{S} := [\widetilde{V}, C^{\rm H}\widetilde{V}, (C^{\rm H})^2\widetilde{V}, \dots, (C^{\rm H})^{M-1}\widetilde{V}].
\end{equation*}
Then, we have
\begin{equation*}
        \widehat{H}_M^\square = \widetilde{S}^{\rm H} \widehat{S}, \quad
        \widehat{H}_M^{\square <} = \widetilde{S}^{\rm H} C \widehat{S}.
\end{equation*}
Therefore, the generalized eigenvalue problem \eqref{eq:hankel_gep} is rewritten as
\begin{equation}
        \widetilde{S}^{\rm H} C \widehat{S} {\bm y}_i = \theta_i \widetilde{S}^{\rm H} \widehat{S}{\bm y}_i.
        \label{eq:hankel_gep2}
\end{equation}
In this form, the block SS--Hankel method can be regarded as a Petrov--Galerkin-type projection method for solving the standard eigenvalue problem \eqref{eq:sep_r}, i.e., the approximate solution $\widetilde{\bm x}_i$ and the corresponding residual ${\bm r}_i:= C \widetilde{\bm x}_i-\theta_i \widetilde{\bm x}_i$ satisfy $\widetilde{\bm x}_i \in \mathcal{R}(\widehat{S})$ and ${\bm r}_i \bot \mathcal{R}(\widetilde{S})$, respectively.
Recognizing that $\mathcal{R}(\widehat{S}) \subset \mathcal{R}(Q_{1:r})$ and applying Theorem~\ref{thm:eigenproblem}, we find that the block SS--Hankel method obtains the target eigenpairs.
\par
Since the Petrov--Galerkin-type projection method for \eqref{eq:sep_r} does not perform the (bi-)orthogonalization; that is $\widetilde{S}^{\rm H} \widehat{S} \neq I$, \eqref{eq:hankel_gep2} describes the generalized eigenvalue problem.
The practical algorithm of the block SS--Hankel method (Algorithm~\ref{alg:ss-hankel}) is derived from a low-rank approximation of \eqref{eq:hankel_gep2}.
\par
From Theorem~\ref{thm:hat_Sk=CkS0}, we have
\begin{equation*}
        \mathcal{R}(\widehat{S}) = \mathcal{K}_M^\square(C,\widehat{S}_0)
\end{equation*}
similar to Theorem~\ref{thm:arnoldi}.
Therefore, the block SS--Arnoldi method can be regarded as a block Arnoldi method with $\mathcal{K}_M^\square(C,\widehat{S}_0)$ for solving the standard eigenvalue problem \eqref{eq:sep_r}.
Moreover, for $M \leq N-\eta$, any $\widehat{E}_M \in \mathcal{B}_M^\square(C,\widehat{S}_0)$ can be written as
\begin{equation*}
        \widehat{E}_M = \sum_{j=1}^N \omega_j \sum_{i=0}^{M-1} z_j^i (z_jB-A)^{-1} B V \alpha_i, \quad
        \alpha_i \in \mathbb{C}^{L \times L}.
\end{equation*}
and the matrix multiplication of $C$ by $\widehat{E}_M$ is given by 
\begin{equation*}
        C \widehat{E}_M = \sum_{j=1}^N \omega_j z_j \sum_{i=0}^{M-1} z_j^i (z_jB-A)^{-1} B V \alpha_i.
\end{equation*}
similar to Theorem~\ref{thm:mat-vec}.
Therefore, in each iteration, the matrix multiplication of $C$ can be performed by a numerical integration.
\par
The Beyn method can be also regarded as a projection method for solving the standard eigenvalue problem \eqref{eq:sep_r}.
From the relation $\widehat{S}_1 = C \widehat{S}_0$ and the singular value decomposition \eqref{eq:beyn_svd} of $\widehat{S}_0$, the coefficient matrix of the eigenvalue problem \eqref{eq:beyn} obtained from the Beyn method becomes
\begin{equation*}
        U_{0,1}^{\rm H} \widehat{S}_1 W_{0,1} \Sigma_1^{-1} = U_{0,1}^{\rm H} C \widehat{S}_0 W_{0,1} \Sigma_{0,1}^{-1} = U_{0,1}^{\rm H} C U_{0,1}.
\end{equation*}
Therefore, the Byen method can be regarded as a Rayleigh--Ritz-type projection method on $\mathcal{R}(U_{0,1})$ for solving \eqref{eq:sep_r}, where $\mathcal{R}(U_{0,1})$ is obtained from a low-rank approximation of $\widehat{S}_0$.
\subsection{Map of the contour integral-based eigensolvers}
As shown in Section 4.1.1, the block SS--RR method and the FEAST eigensolver are based on the Rayleigh--Ritz procedure, which solve the generalized eigenvalue problem $A{\bm x}_i = \lambda_i B {\bm x}_i$.
These methods use subspaces $\mathcal{R}(\widehat{S})$ and $\mathcal{R}(\widehat{S}_0)$, respectively.
The FEAST eigensolver can be regarded as a simplified algorithm of the block SS--RR method with $M=1$ and no orthogonalization of the basis.
Instead, the FEAST eigensolver presupposes an iteration based on an accelerated subspace iteration.
Here, we note that the block SS--RR method can also use an iteration technique for improving accuracy as demonstrated in \cite{Imakura:2015, Sakurai:2013}.
\par
In contrast, as shown in Section 4.1.2, the block SS--Hankel, block SS--Arnoldi and Beyn methods can be regarded as projection methods for solving the standard eigenvalue problem \eqref{eq:sep_r} instead of $A{\bm x}_i=\lambda_i B {\bm x}_i$.
The block SS--Hankel method is a Petrov--Galerkin-type method with $\mathcal{R}(\widehat{S})$, the block SS--Arnoldi method is a block Arnoldi method with $\mathcal{R}(\widehat{S})=\mathcal{K}_M^\square(C,\widehat{S}_0)$ and the Beyn method is a Rayleigh--Ritz-type method with $\mathcal{R}(\widehat{S}_0)$.
Note that because these methods are based on Theorems~\ref{thm:hat_Sk=CkS0} and \ref{thm:eigenproblem}, $(z_j,\omega_j)$ should satisfy condition \eqref{eq:weights}.
\par
Since the block SS--Hankel, block SS--RR and block SS--Arnoldi methods use $\mathcal{R}(\widehat{S})$ as the subspace, the maximum dimension of the subspace is $LM$.
In contrast, the FEAST eigensolver and the Beyn method use the subspace $\mathcal{R}(\widehat{S}_0)$, whose maximum dimension is $L$; that is, ${\rm rank}(\widehat{S}_0)$ can not be larger than the number $L$ of right-hand sides of linear systems at each quadrature point.
Therefore, for the same subspace dimension, the FEAST eigensolver and the Beyn method should incur larger computational costs than the block SS--Hankel, block SS--RR and block SS--Arnoldi methods for solving the linear systems with multiple right-hand sides.
\par
A map of the relationship among the contour integral-based eigensolvers is presented in Fig.~\ref{fig:map}.
\begin{figure}
 \centering
 \includegraphics[scale=0.45, bb=0 70 720 540]{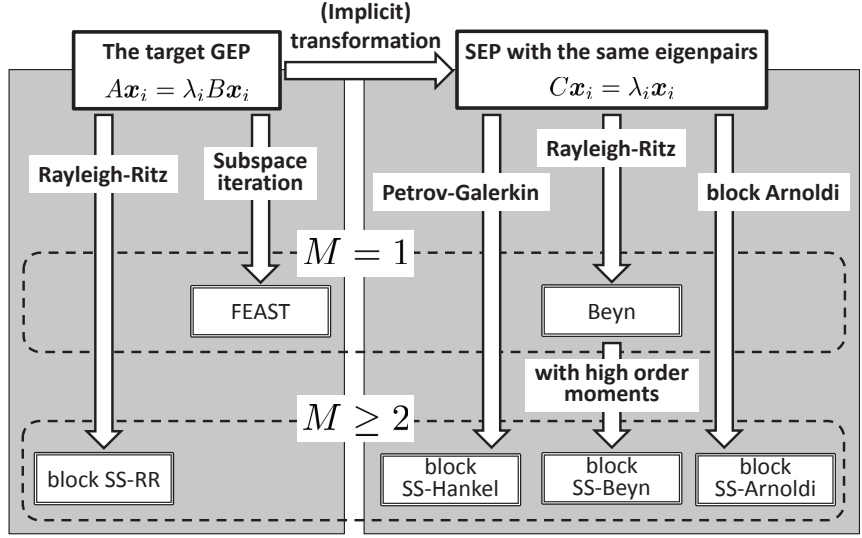}
 \caption{A map of the relationship among the contour integral-based eigensolvers.}
 \label{fig:map}
\end{figure}
\subsection{Proposal for a block SS--Beyn method}
As mentioned above, one iteration of the FEAST eigensolver is a simplified version of the block SS--RR method with $M=1$ and no orthogonalization.
In contrast, a derivative of the Beyn method with $M \geq 2$ has not been proposed.
Although this paper mainly aims to analyze the relationships among these methods and provide a map, we also propose an extension of the Beyn method to $M \geq 2$ as with the block SS--Hankel, block SS--RR and block SS--Arnoldi methods.
\par
As shown in Section~2.2.5, from the relation $\widehat{S}_1 = C \widehat{S}_0$ and a singular value decomposition of $\widehat{S}_0$, we can derive a small size eigenvalue problem \eqref{eq:beyn} of the Beyn method.
As shown in Section~4.1.2, the Beyn method can be also regarded as the Rayleigh--Ritz projection method with $\mathcal{R}(\widehat{S}_0)$ for solving the standard eigenvalue problem \eqref{eq:sep_r}.
To extend the Beyn method, here we consider the Rayleigh--Ritz projection method with $\mathcal{R}(\widehat{S})$ for solving \eqref{eq:sep_r}, i.e.,
\begin{equation*}
        U^{\rm H} C U {\bm t}_i = \theta_i {\bm t}_i
\end{equation*}
where $\widehat{S} = U \Sigma W^{\rm H}$ is a singular value decomposition of $\widehat{S}$.
Using Theorem~\ref{thm:hat_Sk=CkS0}, the coefficient matrix $U^{\rm T} C U$ is replaced as
\begin{equation*}
        U^{\rm H} C U = U^{\rm H} C \widehat{S} W_1 \Sigma_1^{-1} =  U^{\rm H} \widehat{S}_+ W_1 \Sigma_1^{-1},
\end{equation*}
where
\begin{equation*}
        \widehat{S}_+ := [\widehat{S}_1, \widehat{S}_2, \ldots, \widehat{S}_{M}] = C \widehat{S}.
\end{equation*}
\par
In practice, we can also use a low-rank approximation of $\widehat{S}$,
\begin{equation*}
        \widehat{S} = [U_1, U_2] \left[
                \begin{array}{ll}
                        \Sigma_1 & O \\
                        O & \Sigma_2
                \end{array}
        \right] \left[
                \begin{array}{ll}
                        W_1^{\rm H} \\
                        W_2^{\rm H}
                \end{array}
        \right] \approx U_1 \Sigma_1 W_1^{\rm H}.
\end{equation*}
Then, the reduced eigenvalue problem becomes
\begin{equation*}
        U_1^{\rm H} \widehat{S}_+ W_1 \Sigma_1^{-1} {\bm t}_i = \theta_i {\bm t}_i.
\end{equation*}
The approximate eigenpairs are obtained as $(\widetilde{\lambda}_i,\widetilde{\bm x}_i) = (\theta_i, U_{1} {\bm t}_i)$.
In this paper, we call this method as the block SS--Beyn method and show it in Algorithm~\ref{alg:ex-beyn}.
%
\par
Both the block SS--RR method and the block SS--Beyn method are Rayleigh--Ritz-type projection methods with $\mathcal{R}(\widehat{S})$.
However, since the methods are targeted at different eigenvalue problems, they have different definitions of the residual vector.
Therefore, these methods mathematically differ when $B \neq I$.
In contrast, the block SS--Arnoldi method and the block SS--Beyn method without a low-rank approximation, i.e., $\widehat{m} = LM$, are mathematically equivalent.
\begin{algorithm}[t]
\caption{A block SS--Beyn method}
\label{alg:ex-beyn}
\begin{algorithmic}[1]
  \REQUIRE $L, M, N \in \mathbb{N}, V \in \mathbb{C}^{n \times L}, (z_j, \omega_j)$ for $j = 1, 2, \dots, N$
  \ENSURE Approximate eigenpairs $(\widetilde{\lambda}_i, \widetilde{\bm x}_i)$ for $i = 1, 2, \dots, \widehat{m}$
  \STATE Compute $\widehat{S}_k = \sum_{j=1}^{N} \omega_j z_j^k (z_jB-A)^{-1} BV$,\\ and set $\widehat{S} = [\widehat{S}_0, \widehat{S}_1, \ldots, \widehat{S}_{M-1}], \widehat{S}_+ = [\widehat{S}_1, \widehat{S}_2, \ldots, \widehat{S}_{M}]$
  \STATE Compute SVD of $\widehat{S}$: $\widehat{S}= [U_1, U_2] [\Sigma_1, O; O, \Sigma_2] [W_1, W_2]^{\rm H}$
  \STATE Compute eigenpairs $(\theta_i, {\bm t}_i)$ of $U_1^{\rm H}\widehat{S}_+W_1 \Sigma^{-1}_1 {\bm t}_i = \theta_i {\bm t}_i$,\\ and compute $(\widetilde{\lambda}_i, \widetilde{\bm x}_i) = (\theta_i, U_1 {\bm t}_i)$ for $i = 1, 2, \dots, \widehat{m}$
\end{algorithmic}
\end{algorithm}
\section{Error analyses of the contour integral-based eigensolvers with an iteration technique}
\label{sec:error}
As shown in Section~2.2.3, the FEAST eigensolver is based on the iteration.
Other iterative contour integral-based eigensolvers have been designed to improve the accuracy \cite{Imakura:2015, Sakurai:2013}.
The basic concept is the iterative computation of the matrix $\widehat{S}_0^{(\ell-1)}$, from the initial matrix $\widehat{S}_0^{(0)} = V$ as follows:
\begin{equation}
        \widehat{S}^{(\nu)}_0 := \sum_{j=1}^{N} \omega_j (z_jB-A)^{-1} B \widehat{S}_0^{(\nu-1)}, \quad
        \nu = 1, 2, \ldots, \ell-1.
        \label{eq:iter1}
\end{equation}
The matrices $\widehat{S}_k^{(\ell)}$ and $\widehat{S}^{(\ell)}$ are then constructed from $\widehat{S}_0^{(\ell-1)}$ as
\begin{equation}
        \widehat{S}^{(\ell)} := [\widehat{S}_0^{(\ell)}, \widehat{S}_1^{(\ell)}, \ldots, \widehat{S}_{M-1}^{(\ell)} ], \quad
        \widehat{S}^{(\ell)}_k := \sum_{j=1}^{N} \omega_j z_j^k (z_jB-A)^{-1} B \widehat{S}_0^{(\ell-1)},
        \label{eq:iter2}
\end{equation}
and $\mathcal{R}( \widehat{S}_0^{(\ell)} )$ and $\mathcal{R}( \widehat{S}^{(\ell)} )$ are used as subspaces rather than $\mathcal{R}( \widehat{S}_0 )$ and $\mathcal{R}( \widehat{S} )$.
The $\ell$ iterations of the FEAST eigensolver can be regarded as a Rayleigh--Ritz-type projection method on $\mathcal{R}(\widehat{S}_0^{(\ell)})$.
\par
From the discussion in Section~\ref{sec:preparation}, the matrix $\widehat{S}_0^{(\ell)}$ can be expressed as
\begin{equation*}
        \widehat{S}_0^{(\ell)} = \left( Q_{1:r} F_{1:r} \widetilde{Q}_{1:r}^{\rm H} \right)^\ell V.
\end{equation*}
Here, the eigenvalues of the linear operator $\widehat{P} := Q_{1:r} F_{1:r} \widetilde{Q}_{1:r}^{\rm H}$ are given by
\begin{equation*}
        f(\lambda_i) := \sum_{j=1}^N \frac{\omega_j}{z_j - \lambda_i}.
\end{equation*}
The function $f(\lambda)$, called the filter function, is used in the analyses of some eigensolvers with diagonalizable matrix pencil \cite{Polizzi:2014, Imakura:2015, Schofield:2012, Guttel:2015}.
The function $f(\lambda)$ is characterized by $|f(\lambda)|\approx 1$ in the inner region and $|f(\lambda)|\approx 0$ in the outer region.
Fig.~\ref{fig:filter} plots the filter function when $\Omega$ is the unit circle and integration is performed by the $N$-point trapezoidal rule.
\begin{figure}[t]
\begin{center}
\subfloat[On the real axis for $N=16, 32, 64$.]{
\includegraphics[bb=0 0 360 216, width=60mm]{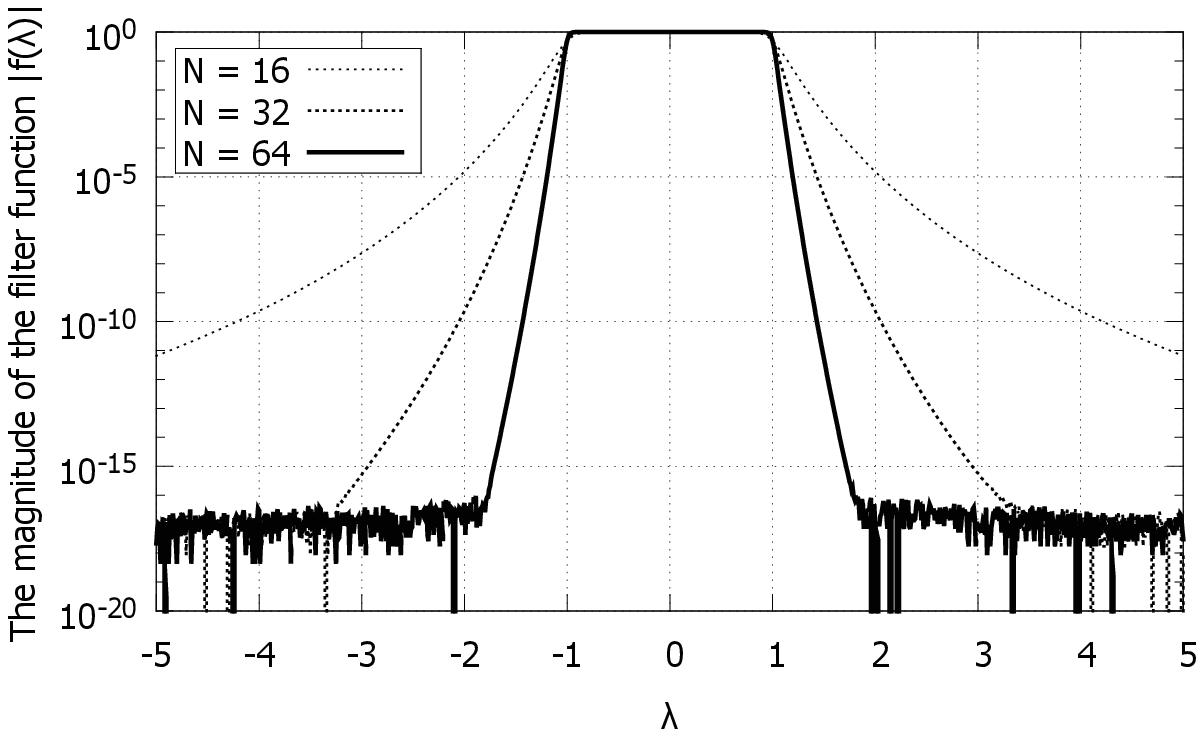}
}
\subfloat[On the complex plane for $N=32$.]{
\includegraphics[bb=0 0 360 216, width=60mm]{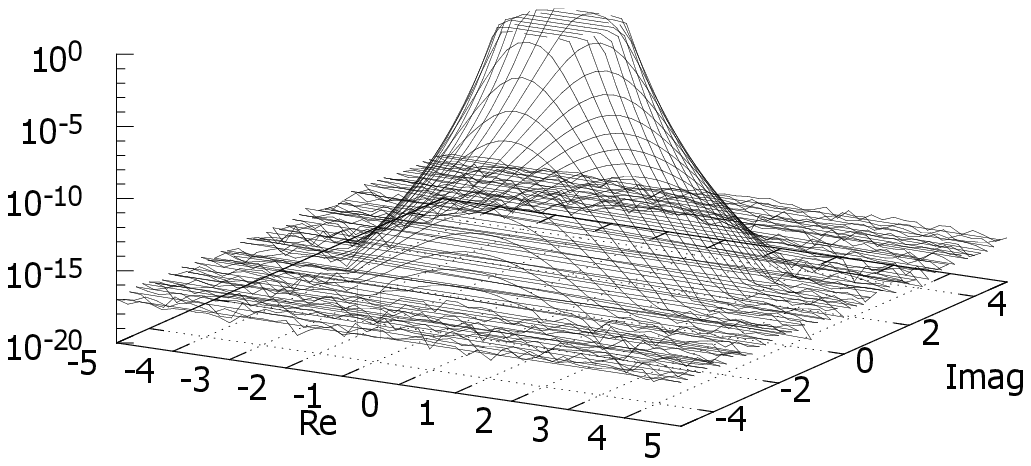}
}
\caption{Magnitude of filter function $|f(\lambda)|$ of the $N$-point trapezoidal rule for the unit circle region $\Omega$.}
\label{fig:filter}
\end{center}
\end{figure}
\par
Error analyses of the block SS--RR method with the iteration technique \eqref{eq:iter1} and \eqref{eq:iter2} and the FEAST eigensolver in the diagonalizable case were given in \cite{Imakura:2015, Polizzi:2014, Guttel:2015}.
In these error analyses, the block SS--RR method and the FEAST eigensolver were treated as projection methods with the subspaces $\mathcal{R}(\widehat{S})$ and $\mathcal{R}(\widehat{S}_0)$, respectively.
In Section~\ref{sec:map}, we explained that the other contour integral-based eigensolvers are also projection methods with the subspaces $\mathcal{R}(\widehat{S})$ and $\mathcal{R}(\widehat{S}_0)$, but were designed to solve the standard eigenvalue problem \eqref{eq:sep_r}.
In this section, we establish the error bounds of the contour integral-based eigensolvers with the iteration technique \eqref{eq:iter1} and \eqref{eq:iter2}, omitting the low-rank approximation, in non-diagonalizable cases.
\subsection{Error bounds of the block SS--RR method and the FEAST eigensolver in the diagonalizable case}
Let $(\lambda_i,{\bm x}_i)$ be exact finite eigenpairs of the generalized eigenvalue problem $A {\bm x}_i = \lambda_i B {\bm x}_i$.
Assume that $f(\lambda_i)$ are ordered by decreasing magnitude $|f(\lambda_i)| \geq |f(\lambda_{i+1})|$.
Define $\mathcal{P}^{(\ell)}$ and $\mathcal{P}_{LM}$ as orthogonal projectors onto the subspaces $\mathcal{R}( \widehat{S}^{(\ell)} )$ and the spectral projector with an invariant subspace ${\rm span}\{ {\bm x}_1, {\bm x}_2, \ldots, {\bm x}_{LM} \}$, respectively.
Assume that the matrix $\mathcal{P}_{LM} [V, CV, \dots, C^{M-1}V]$ is full rank.
Then, for each eigenvector ${\bm x}_i, i = 1, 2, \ldots, LM$, there exists a unique vector ${\bm s}_i \in \mathcal{K}_M^\square(C,V)$ such that $\mathcal{P}_{LM} {\bm s}_i = {\bm x}_i$.
\par
In the diagonalizable case, for the error analysis of the block SS--RR method and the FEAST eigensolver, the following inequality was given in \cite{Imakura:2015} and \cite{Polizzi:2014, Guttel:2015} for $M=1$:
\begin{equation}
        \| (I - \mathcal{P^{(\ell)}} ) {\bm x}_i \|_2
        \leq \alpha \beta_i \left| \frac{f(\lambda_{LM+1})}{f(\lambda_i)} \right|^\ell, \quad
        i = 1, 2, \ldots, LM,
        \label{eq:ineq1}
\end{equation}
where $\alpha = \| X_r \|_2 \| \widetilde{X}_r \|_2$ and $\beta_i = \| {\bm x}_i - {\bm s}_i \|_2$.
Note that, in the diagonalizable case, the linear operator $\widehat{P}$ can be expressed as $\widehat{P} = X_r f(\Lambda_r) \widetilde{X}_r^{\rm H}$, where $f(\Lambda_r) := {\rm diag}(f(\lambda_1), f(\lambda_2), \dots f(\lambda_r))$.
An additional error bound is given in \cite{Imakura:2015}:
\begin{equation}
        \| (A_{\mathcal{P}^{(\ell)}} - \lambda_i B_{\mathcal{P}^{(\ell)}} ) {\bm x}_i \|_2
        \leq \gamma_i \| (I - \mathcal{P^{(\ell)}} ) {\bm x}_i \|_2 
        \leq \alpha \beta_i \gamma_i \left| \frac{f(\lambda_{LM+1})}{f(\lambda_i)} \right|^\ell,
        \label{eq:ineq2}
\end{equation}
for $i = 1, 2, \ldots, LM$, where $A_{\mathcal{P}^{(\ell)}} := \mathcal{P}^{(\ell)} A \mathcal{P}^{(\ell)}, B_{\mathcal{P}^{(\ell)}} := \mathcal{P}^{(\ell)} B \mathcal{P}^{(\ell)}$ and $\gamma_i = \| \mathcal{P}^{(\ell)} (A - \lambda_i B) (I - \mathcal{P}^{(\ell)}) \|_2$.
\par
Inequality \eqref{eq:ineq1} determines the accuracy of the subspace $\mathcal{R}(\widehat{S})$, whereas inequality \eqref{eq:ineq2} defines the error bound of the block SS--RR method and the FEAST eigensolver.
\subsection{Error bounds of the contour integral-based eigensolvers in the non-diagonalizable case}
The constant $\alpha$ in \eqref{eq:ineq1} derives from the following inequality for a diagonalizable matrix $G_{\rm diag} = X D X^{-1}$
\begin{equation*}
        \| G_{\rm diag}^\ell \|_2 
        \leq \| X \|_2 \| D^\ell \|_2 \| X^{-1} \|_2
        \leq \| X \|_2 \| X^{-1} \|_2 (\rho(G_{\rm diag}))^\ell,
\end{equation*}
where $\rho(G_{\rm diag})$ is the spectral radius of $G_{\rm diag}$.
This inequality is extended to a non-diagonalizable matrix $D_{\rm non} = X J X^{-1}$ as follows:
\begin{equation*}
        \| G_{\rm non}^\ell \|_2 
        \leq \| X \|_2 \| J^\ell \|_2 \| X^{-1} \|_2
        \leq 2 \| X \|_2 \| X^{-1} \|_2 \ell^{\eta-1} (\rho(G_{\rm non}))^\ell,
\end{equation*}
where $\rho(G_{\rm non})$ is the spectral radius of $G_{\rm non}$ and $\eta$ is the maximum size of the Jordan blocks.
Using this inequality, the error bound of the contour integral-based eigensolvers in the non-diagonalizable case is given as
\begin{equation}
        \| (I - \mathcal{P^{(\ell)}} ) {\bm x}_i \|_2
        \leq \alpha^\prime \beta_i \ell^{\eta-1} \left| \frac{f(\lambda_{LM+1})}{f(\lambda_i)} \right|^\ell, \quad
        i = 1, 2, \ldots, LM,
        \label{eq:ineq3}
\end{equation}
where $\alpha^\prime = 2 \| Q_{1:r} \|_2 \| \widetilde{Q}_{1:r} \|_2$.
From \eqref{eq:ineq3}, the error bound of the block SS--RR method and the FEAST eigensolver in the non-diagonalizable case is given by
\begin{equation}
        \| (A_{\mathcal{P}^{(\ell)}} - \lambda_i B_{\mathcal{P}^{(\ell)}} ) {\bm x}_i \|_2
        \leq \gamma_i \| (I - \mathcal{P^{(\ell)}} ) {\bm x}_i \|_2
        \leq \alpha^\prime \beta_i \gamma_i \ell^{\eta-1} \left| \frac{f(\lambda_{LM+1})}{f(\lambda_i)} \right|^\ell,
        \label{eq:ineq_rf}
\end{equation}
for $i = 1, 2, \dots LM$.
\par
The inequality \eqref{eq:ineq_rf} derives from the error bound of the Rayleigh--Ritz procedure for generalized eigenvalue problems $A{\bm x}_i = \lambda_i B {\bm x}_i$.
From the error bound of the Rayleigh--Ritz procedure for standard eigenvalue problems \cite[Theorem~4.3]{Saad:2011}, we derive the error bound of the block SS--Arnoldi and block SS--Beyn methods as
\begin{equation}
        \| (C_{\mathcal{P}^{(\ell)}} - \lambda_i I ) \mathcal{P}^{(\ell)}{\bm x}_i \|_2
        \leq \gamma^\prime \| (I - \mathcal{P^{(\ell)}} ) {\bm x}_i \|_2
        \leq \alpha^\prime \beta_i \gamma^\prime \ell^{\eta-1} \left| \frac{f(\lambda_{LM+1})}{f(\lambda_i)} \right|^\ell,
        \label{eq:ineq_ab}
\end{equation}
for $i = 1, 2, \ldots, LM$, where $C_{\mathcal{P}^{(\ell)}} := \mathcal{P}^{(\ell)} C \mathcal{P}^{(\ell)}$ and $\gamma^\prime = \| \mathcal{P}^{(\ell)} C (I - \mathcal{P}^{(\ell)}) \|_2$.
\par
In addition, let $\mathcal{Q}$ be the oblique projector onto $\mathcal{R}(\widehat{S}^{(\ell)})$ and orthogonal to $\mathcal{R}( \widetilde{S})$.
Then, from the error bound of the Petrov--Galerkin-type projection method for standard eigenvalue problems \cite[Theorem~4.7]{Saad:2011}, the error bound of the block SS--Hankel method is derived as follows:
\begin{equation}
        \| (C_{\mathcal{P}^{(\ell)}}^\mathcal{Q} - \lambda_i I ) \mathcal{P}^{(\ell)}{\bm x}_i \|_2
        \leq \gamma_i^{\prime \prime} \| (I - \mathcal{P^{(\ell)}} ) {\bm x}_i \|_2
        \leq \alpha^\prime \beta_i \gamma_i^{\prime \prime} \ell^{\eta-1} \left| \frac{f(\lambda_{LM+1})}{f(\lambda_i)} \right|^\ell,
        \label{eq:ineq_h}
\end{equation}
for $i = 1, 2, \dots, LM$, where $C_{\mathcal{P}^{(\ell)}}^\mathcal{Q} := \mathcal{Q} C \mathcal{P}^{(\ell)}$ and $\gamma_i^{\prime \prime} = \| \mathcal{Q} (C - \lambda_i I) (I - \mathcal{P}^{(\ell)}) \|_2$.
\par
Error bounds \eqref{eq:ineq_rf}, \eqref{eq:ineq_ab} and \eqref{eq:ineq_h} indicate that given a sufficiently large subspace, i.e., $|f(\lambda_{LM+1})/f(\lambda_i)|^\ell \approx 0$, the contour integral-based eigensolvers can obtain the accurate target eigenpairs even if some eigenvalues exist outside but near the region and the target matrix pencil is non-diagonalizable.
\section{Numerical experiments}
\label{sec:experiments}
This paper mainly aims to analyze the relationships among the contour integral-based eigensolvers and to map these relationships; although, in this section, the efficiency of the block SS--Hankel, block SS--RR, block SS--Arnoldi and block SS--Beyn methods are compared in numerical experiments with $M = 1, 2, 4, 8$ and $16$.
\par
These methods compute 1000 eigenvalues in the interval $[-1,1]$ and the corresponding eigenvectors of a real symmetric generalized eigenvalue problem with 20000 dimensional dense and random matrices.
$\Gamma$ is an ellipse with center 0 and major and minor axises 1 and 0.1, respectively.
The parameters are $(L,M) = (4096,1), (2048,2), (1024,4)$, $(512,8)$, $(256,16)$ (note that $LM=4096$) and $N=32$.
Because of a symmetry of the problem, the number of required linear systems is $N/2=16$.
For the low-rank approximation, we used singular values $\sigma_i$ satisfying $\sigma_i/\sigma_1 \geq 10^{-14}$ and their corresponding singular vectors, where $\sigma_1$ is the largest singular value.
\par
The numerical experiments were carried out in double precision arithmetic on 8 nodes of COMA at CCS, University of Tsukuba.
COMA has two Intel Xeon E5-2670v2 (2.5 GHz) and two Intel Xeon Phi 7110P (61 cores) per node.
In these numerical experiments, we used only the CPU part.
The algorithms were implemented in Fortran 90 and MPI, and executed with 8 [node] $\times$ 2 [process/node] $\times$ 8 [thread/process].
\begin{table}[!t]
\small
\caption{Computational results of the block SS--Hankel, block SS--RR, block SS--Arnoldi and block SS--Beyn methods with $M=1, 2, 4, 8$ and $16$.}
\label{table:results}
\begin{center}
\begin{tabular}{rrrrrrrrrr} \hline
 \multicolumn{1}{c}{$M$} & \multicolumn{1}{c}{$\widehat{m}$} & & \multicolumn{4}{c}{Time [sec.]} & & \multicolumn{2}{c}{residual norm} \\\cline{4-7} \cline{9-10}
 \multicolumn{1}{c}{} & \multicolumn{1}{c}{} & & \multicolumn{1}{c}{$t_{\rm LU}$} & \multicolumn{1}{c}{$t_{\rm Solve}$} & \multicolumn{1}{c}{$t_{\rm Other}$} & \multicolumn{1}{c}{$t_{\rm Total}$} & & \multicolumn{1}{c}{\footnotesize $\max_{\lambda_i \in \Omega}\|{\bm r}_i\|_2$} & \multicolumn{1}{c}{\footnotesize $\min_{\lambda_i \in \Omega}\|{\bm r}_i\|_2$} \\\hline \hline
 \\
 \multicolumn{10}{l}{block SS--Hankel method} \\ \hline
 1  & 1274 & & 126.47  & 97.80  & 41.57  & 265.84  & & $1.72 \times 10^{-14}$ & $3.06 \times 10^{-15}$ \\
 2  & 1291 & & 126.38  & 49.02  & 28.74  & 204.14  & & $1.12 \times 10^{-12}$ & $2.72 \times 10^{-15}$ \\
 4  & 1320 & & 126.46  & 25.40  & 25.93  & 177.78  & & $2.15 \times 10^{-14}$ & $3.16 \times 10^{-15}$ \\
 8  & 1419 & & 126.33  & 13.53  & 26.39  & 166.25  & & $1.31 \times 10^{-11}$ & $1.66 \times 10^{-14}$ \\
 16 & 2206 & & 126.24  &  7.65  & 32.41  & 166.30  & & $1.64 \times 10^{-06}$ & $1.59 \times 10^{-11}$ \\\hline
 \\
 \multicolumn{10}{l}{block SS--RR method} \\ \hline
 1  & 1283 & & 126.45  & 97.27  & 38.62  & 262.33  & & $1.34 \times 10^{-13}$ & $1.05 \times 10^{-13}$ \\
 2  & 1292 & & 126.31  & 48.77  & 38.84  & 213.92  & & $1.35 \times 10^{-13}$ & $9.56 \times 10^{-14}$ \\
 4  & 1304 & & 126.34  & 25.22  & 38.49  & 190.05  & & $1.73 \times 10^{-13}$ & $9.89 \times 10^{-14}$ \\
 8  & 1340 & & 126.33  & 13.46  & 38.78  & 178.57  & & $5.53 \times 10^{-13}$ & $1.16 \times 10^{-13}$ \\
 16 & 1461 & & 126.49  &  7.65  & 40.84  & 174.98  & & $1.34 \times 10^{-11}$ & $1.24 \times 10^{-13}$ \\\hline
 \\
 \multicolumn{10}{l}{block SS--Arnoldi method} \\ \hline
 1  & 4096 & & 125.96  & 97.13  & 94.58  & 317.66  & & $4.72 \times 10^{-08}$ & $4.46 \times 10^{-12}$ \\
 2  & 4096 & & 126.43  & 48.84  & 62.11  & 237.37  & & $5.24 \times 10^{-08}$ & $1.99 \times 10^{-13}$ \\
 4  & 4096 & & 126.13  & 25.20  & 52.61  & 203.94  & & $2.64 \times 10^{-08}$ & $5.24 \times 10^{-13}$ \\
 8  & 4096 & & 126.23  & 13.46  & 49.32  & 189.02  & & $9.05 \times 10^{-09}$ & $8.80 \times 10^{-13}$ \\
 16 & 4096 & & 126.35  &  7.63  & 54.41  & 188.38  & & $9.31 \times 10^{-07}$ & $7.70 \times 10^{-13}$ \\\hline
 \\
 \multicolumn{10}{l}{block SS--Beyn method} \\ \hline
 1  & 1283 & & 126.17  & 97.24  & 32.63  & 256.05  & & $1.34 \times 10^{-13}$ & $1.06 \times 10^{-13}$ \\
 2  & 1292 & & 126.48  & 48.76  & 32.14  & 207.37  & & $1.36 \times 10^{-13}$ & $9.58 \times 10^{-14}$ \\
 4  & 1304 & & 126.22  & 25.22  & 31.25  & 182.69  & & $1.74 \times 10^{-13}$ & $9.91 \times 10^{-14}$ \\
 8  & 1340 & & 126.21  & 13.44  & 31.09  & 170.74  & & $5.54 \times 10^{-13}$ & $1.16 \times 10^{-13}$ \\
 16 & 1461 & & 126.45  &  7.65  & 32.25  & 166.35  & & $1.90 \times 10^{-10}$ & $1.25 \times 10^{-13}$ \\\hline
\end{tabular}
\end{center}
\end{table}
\par
The numerical results are presented in Table~\ref{table:results}.
First, we consider the numerical rank $\widehat{m}$.
Comparing $M$ dependence of the numerical rank $\widehat{m}$ in the block SS--Hankel, block SS--RR and block SS--Beyn methods, we observe that the numerical rank $\widehat{m}$ increases with increasing $M$.
This is causally related to the property of the subspace $\mathcal{K}_M^\square(C,V)$, because $\widehat{S}$ is written as
\begin{equation*}
        \widehat{S} = \left( Q_{1:r} F_{1:r} \widehat{Q}_{1:r}^{\rm H} \right) [V, CV, \dots, C^{M-1}V].
\end{equation*}
For $M=1$, the subspace $\mathcal{K}_1^\square(C,V) = \mathcal{R}(V)$ is unbiased for all eigenvectors, since $V$ is a random matrix.
On the other hand, for $M \geq 2$, the subspace $\mathcal{K}_M^\square(C,V)$ contains eigenvectors corresponding exterior eigenvalues well.
Therefore, for computing interior eigenvalues, the numerical rank $\widehat{m}$ for $M=16$ is expected to be larger than for $M=1$.
%
\par
Next, we consider the computation time.
The computation times of the LU factorization, forward and back substitutions and the other computation time including the singular value decomposition and orthogonalization are denoted by $t_{\rm LU}, t_{\rm Solve}, t_{\rm Other}$, respectively.
The total computation time is also denoted by $t_{\rm Total}$.
We observe, from Table~\ref{table:results}, that the most time-consuming part is to solve linear systems with multiple-right hand sides ($t_{\rm LU} + t_{\rm Solve}$).
In particular, $t_{\rm Solve}$ is much larger for $M=1$ than for $M=16$, because the number of right-hand sides for $M=1$ is 16 times larger than for $M=16$.
Consequently, $t_{\rm Total}$ increases with decreasing $M$.
\par
We now focus on $t_{\rm Other}$.
The block SS--Arnoldi method consumes much greater $t_{\rm Other}$ than the other methods because its current version applies no low-rank approximation technique to reduce the computational costs and improve the stability \cite{Imakura:2014}.
For the block SS--Hankel, block SS--RR and block SS--Beyn methods, $t_{\rm Other}$ is smaller as $M$ and the numerical rank $\widehat{m}$ are smaller.
In addition, the block SS--Hankel method consumes smallest $t_{\rm Other}$ among tested methods, because it performs no matrix orthogonalization.
\par
Finally, we consider the accuracy of the computed eigenpairs.
The block SS--Hankel and block SS--Arnoldi methods are less accurate than the other methods, specifically for $M=16$.
This result is attributed to no matrix orthogonalization in the block SS--Hankel method, and to no low-rank approximation in the block SS--Arnoldi method.
On the other hand, the block SS--RR and block SS--Beyn methods show high accuracy even for $M=16$.
\section{Conclusions}
\label{sec:conclusions}
In this paper, we analyzed and mapped the mathematical relationships among the algorithms of the typical contour integral-based methods for solving generalized eigenvalue problems \eqref{eq:gep}: the block SS--Hankel method, the block SS--RR method, the FEAST eigensolver, the block SS--Arnoldi method and the Beyn method.
We found that the block SS--RR method and the FEAST eigensolver are projection methods for $A{\bm x}_i=\lambda_i B{\bm x}_i$, whereas the block SS--Hankel, block SS--Arnoldi and Beyn methods are projection methods for the standard eigenvalue problem $C{\bm x}_i=\lambda_i {\bm x}_i$.
From the map of the algorithms, we also extended the existing Beyn method to $M \geq 2$.
Our numerical experiments indicated that increasing $M$ reduces the computational costs (relative to $M=1$).
\par
In future, we will compare the efficiencies of these methods in solving large, real-life problems.
We also plan to analyze the relationships among contour integral-based nonlinear eigensolvers.
\section*{Acknowledgements}
The authors would like to thank Dr. Kensuke Aishima, The University of Tokyo for his valuable comments.
The authors are also grateful to an anonymous referee for useful comments.
%

\end{document}